\documentclass[11pt]{amsart}
\usepackage{times,epsf,amssymb,amsmath,hyperref, rlepsf, color}
\usepackage{amsmath}
\usepackage{dsfont}
\usepackage{color}
\usepackage{graphicx}
\usepackage{amsbsy}
\usepackage{verbatim}
\usepackage{amsthm}

 \DeclareMathOperator{\spt}{spt} \DeclareMathOperator{\inr}{int} \DeclareMathOperator{\dist}{dist}
  
 \DeclareMathOperator{\genus}{genus}
  
    \DeclareMathOperator{\area}{Area}

  \def\e{\epsilon}
  \def\a{\alpha}
    \def\d{\delta}

\def\mass{\underline{\underline{M}}}
\def\wt{\widetilde}
\def\wh{\widehat}

\def\ov{\overline}

\def\R{\mathbf{R}}
\def\N{\mathbb{N}}

\begin{document}

\newtheorem{thm}{Theorem}[section]
\newtheorem{lem}[thm]{Lemma}
\newtheorem{cor}[thm]{Corollary}

\theoremstyle{remark}
\newtheorem{rmk}[thm]{Remark}
\newtheorem{defn}[thm]{\bf{Definition}}

\def\square{\hfill${\vcenter{\vbox{\hrule height.4pt \hbox{\vrule
width.4pt height7pt \kern7pt \vrule width.4pt} \hrule height.4pt}}}$}

\def\T{\mathcal T}

\newenvironment{pf}{{\it Proof:}\quad}{\square \vskip 12pt}

\title{Area minimizing surfaces in mean convex 3-manifolds}
\author{Theodora Bourni}
\author{Baris Coskunuzer}
\address{Frei University, Berlin, Germany}
\email{bourni@math.fu-berlin.de}
\address{Koc University \\ Department of Mathematics \\ Sariyer, Istanbul 34450 Turkey}
\email{bcoskunuzer@ku.edu.tr}
\thanks{The second author is partially supported by EU-FP7 Grant IRG-226062, TUBITAK Grant 109T685 and TUBA-GEBIP Award.}

\maketitle


\newcommand{\BR}{\mathbf{R}}
\newcommand{\BRR}{\mathbf{R}^3}
\newcommand{\BC}{\mathbf{C}}
\newcommand{\BZ}{\mathbf{Z}}
\newcommand{\BN}{\mathbf{N}}
\newcommand{\BB}{\mathbf{B}}
\newcommand{\A}{\mathcal{A}}
\newcommand{\B}{\mathcal{B}}
\newcommand{\U}{\mathcal{U}}
\newcommand{\C}{\mathcal{C}}
\newcommand{\CC}{\mathfrak{C}}

\begin{abstract}

In this paper, we give several results on area minimizing surfaces in strictly mean convex $3$-manifolds. First, we study the genus of absolutely
area minimizing surfaces in a compact, orientable, strictly mean convex $3$-manifold $M$ bounded by a simple closed curve in $\partial M$. Our main
result is that for any $g\ge 0$, the space of simple closed curves in $\partial M$ where all the absolutely area minimizing surfaces they bound in
$M$ has genus $\geq g$ is open and dense in the space $\mathcal{A}$ of nullhomologous simple closed curves in $\partial M$. For showing this we prove a bridge principle for absolutely area minimizing surfaces. Moreover, we
show that for any $g\ge 0$, there exists a curve in $\A$ such that the minimum genus of the absolutely area minimizing surfaces it bounds is exactly
$g$.

As an application of these results, we further prove that the simple closed curves in $\partial M$ bounding more than one minimal surface in $M$ is
an open and dense subset of $\mathcal{A}$. We also show that there are disjoint simple closed curves in $\partial M$ bounding minimal surfaces in $M$
which are not disjoint. This allows us to answer a question of Meeks, by showing that for any strictly mean convex $3$-manifold $M$, there exists a
simple closed curve $\Gamma$ in $\partial M$ which bounds a stable minimal surface which is not embedded.

\end{abstract}

\section{Introduction}

The Plateau problem concerns the existence of an area minimizing disk with boundary a given curve in an ambient manifold $M$. This problem was solved in the
case when the ambient manifold is $\BR^3$ by Douglas \cite{Do}, and Rado \cite{Ra} in the early 1930s. Later, it was generalized by Morrey \cite{Mo} for
Riemannian manifolds. In the 1980s, Meeks and Yau showed that if $M$ is a mean convex $3$-manifold, and $\Gamma$ is a simple closed curve in $\partial
M$, then any area minimizing disk with boundary $\Gamma$ is embedded \cite{MY1}. Later, White gave a generalization of this result to any genus
\cite{Wh2}.

In the early 1960s, the same question was studied for absolutely area minimizing surfaces, i.e. for surfaces that minimize area among all orientable
surfaces with a given boundary (without restriction on the genus). Using techniques from geometric measure theory, Federer and Fleming \cite{FeF}
were able to solve this problem by proving the existence of an absolutely area minimizing integral current (see also \cite{DG, Giu} for the existence
of a minimizing Caccioppoli set). In \cite{ASSi} Almgren, Schoen and Simon showed that this current is a smooth embedded surface away from its
boundary and Hardt \cite{H} showed that it is also smooth at the boundary, provided that the prescribed boundary is smooth and lies on a convex set.
Later, Hardt and Simon \cite{HSi}, improved this boundary regularity result, by dropping the assumption that  the prescribed boundary lies on a
convex set.

It can be seen that there are two main versions of the Plateau problem; one of them is that of fixed genus (area minimizing in a fixed topological
class), and the other one is with no restriction on the genus (absolutely area minimizing case). There have also been many important results on the
{\em a priori} bounds on the genus of an absolutely area minimizing surface bounded by a given simple closed curve \cite{HSi}.

In this paper, we study the genus of  absolutely area minimizing surfaces $\Sigma$ in a compact, orientable, strictly mean convex $3$-manifold $M$,
with boundary $\partial \Sigma$  a simple closed curve lying in $\partial M$. We consider the stratification of the space of nullhomologous (bounding
a surface in $M$) simple closed curves in $\partial M$ with respect to the minimum genus of the absolutely area minimizing surfaces that they bound.
In particular, we let $\mathcal{A}$ be the space of all nullhomologous simple closed curves in $\partial M$ with the $\mathcal C^0$ topology (see
Definition \ref{c0rmk}).  By the previously mentioned results, any $\Gamma\in \mathcal {A}$ bounds an embedded absolutely area minimizing surface
$\Sigma$ in $M$ with $\partial \Sigma = \Gamma$. Let $\mathcal{A}_g$ be the set of all the curves in $\mathcal{A}$ such that any embedded absolutely
area minimizing surface in $M$ that they bound has genus $\geq g$. Then, clearly $\A=\A_0\supset \A_1\supset \A_2 \supset ...\supset \A_n \supset
...$. We show that for any $g\geq 0$, $\A_g$ is an open subset of $\A$ (Lemma \ref{Agopen}).

Then, we define an operation, which we call {\em horn surgery}, on a given simple closed curve $\Gamma$ in $\partial M$ which adds a {\em thin
handle} to an absolutely area minimizing surface $\Sigma$ which $\Gamma$ bounds in $M$. In other words, by modifying the boundary curve, the
operation increases the genus of an absolutely area minimizing surface (See Figure 4 and Figure 5). Moreover, this new curve can be made as close as
we want to the original curve. Hence by using the horn surgery operation, it is easy to show that $\A_g$ is not only open, but also dense in $\A$
(Theorem \ref{Adense}). On the other hand, to define the horn surgery operation, we prove {\em a bridge principle for absolutely area minimizing
surfaces} (Lemma \ref{bridge principle}).

Next, we study the the space $\B_g$ of all the curves in $\mathcal{A}$, such that the minimum genus of the embedded absolutely area minimizing surfaces in
$M$ that they bound is exactly $g$, i.e. $\B_g=\A_g \setminus \A_{g+1}$. Then, clearly $\A = \bigcup_{g=0}^{\infty} \B_g$ and $\B_g\cap \B_{g'} =
\emptyset$ for any $g\neq g'$. Again by using the horn surgery operation, we show that for any $g\geq 0$,  $\B_g$ is not empty (Theorem \ref{Bgne}).
In other words, we show that for any $g\geq 0$, there is a simple closed curve $\Gamma_g$ in $\partial M$ that bounds an absolutely area minimizing
surface $\Sigma$ with genus $g$. Also, we prove that for any $g$, $\B_g$ is nowhere dense in $\A$ (Corollary \ref{Bndense}). Furthermore, we show
that there exist simple closed curves $\Gamma$ in $\partial M$ which bound two absolutely area minimizing surfaces of different genus (Theorem
\ref{diffgenus}).

In Section \ref{minsurfaces section}, we further derive several interesting results about minimal surfaces in strictly mean convex $3$-manifolds by
using the previously mentioned results. We show that some important properties for properly embedded area minimizing surfaces in such manifolds are not valid for
properly embedded minimal surfaces. First, we show that curves that bound more that one minimal surface are generic for a strictly mean convex
$3$-manifold $M$ (Theorem \ref{nonuniquecurves}). In \cite{Co1}, and \cite{CE}, it is proven that a generic nullhomotopic simple closed curve in
$\partial M$ bounds a unique area minimizing disk, and similarly, a generic nullhomologous simple closed curve in $\partial M$ bounds a unique
absolutely area minimizing surface in $M$. However, here we show that when we relax the condition of being area minimizing to just minimal, the
situation is completely opposite.

Finally, we generalize Peter Hall's results \cite{Ha}, which answers Meeks' questions, to any strictly mean convex $3$-manifold $M$. As previously
mentioned, in \cite{MY1}, Meeks and Yau proved that any area minimizing disk in a mean convex $3$-manifold bounded by a simple closed curve in
$\partial M$ must be embedded . After establishing this result, Meeks posed the question of whether or not the same holds for stable minimal
surfaces. Then, Hall constructed an example of a simple closed curve $\Gamma$ in $S^2=\partial \BB^3$, where $\BB^3$ is the unit $3$-ball in $\R^3$,
such that $\Gamma$ bounds a stable minimal disk $M$ in $B^3$ which is not embedded. This shows that if we relax the area minimizing condition to just
being minimal again, the embeddedness result of Meeks and Yau is no longer valid. We generalize Hall's example, and show that this is true for any
strictly mean convex $3$-manifold (Theorem \ref{nonembeg}).

To construct the nonembedded minimal examples, we first show that for any strictly mean convex $3$-manifold $M$, there are disjoint simple closed
curves $\Gamma_1$ and $\Gamma_2$ in $\partial M$ that bound  minimal surfaces, $M_1$ and $ M_2$ respectively, in $M$, such that $M_1\cap M_2 \neq
\emptyset$ (Theorem \ref{intersectingminsurfaces}). This is very interesting since if $M$ is a mean convex $3$-manifold with trivial second homology,
and $\Gamma_1,\Gamma_2$ are two disjoint simple closed curves in $\partial M$,  then any two absolutely area minimizing surfaces $\Sigma_1, \Sigma_2$
with $\partial \Sigma_1 = \Gamma_1,  \partial \Sigma_2 = \Gamma_2$ are also disjoint \cite{Co1}. The same holds for area minimizing disks, too.
Hence, again the case of minimal surfaces, and that of area minimizing are very different. Then, by connecting these intersecting minimal surfaces,
which have disjoint boundaries, with a {\em bridge}, we construct  nonembedded examples of minimal surfaces in the strictly mean convex $3$-manifold
$M$. Thus, we show that for any such $M$, there exists a simple closed curve $\Gamma$ in $\partial M$ which bounds a stable minimal surface, and in
particular a stable minimal disk, in $M$, which is not embedded (Theorem \ref{nonembeg}).

\subsection*{Acknowledgements:}

We would like to thank Tolga Etgu and Brian White for very useful conversations.

\section{Preliminaries}

In this section we give some standard definitions as well as an overview of the basic known results which we use in the following sections.

\begin{defn} An {\em area minimizing disk} is a disk which has the smallest area among all disks with the same boundary. An {\em area minimizing surface}
is a surface which has the smallest area among all surfaces with the same genus and the same boundary. An {\em absolutely area minimizing surface} is
a surface which has the smallest area among all orientable surfaces (with no topological restriction) with the same boundary.
\end{defn}

\begin{defn} Let $M$ be a compact Riemannian $3$-manifold with boundary. Then $M$ is called {\em mean convex} (or sufficiently convex) if the following conditions hold:

\begin{itemize}

\item[(i)] $\partial M$ is piecewise smooth.

\item[(ii)] Each smooth subsurface of $\partial M$ has nonnegative mean curvature with respect to the inward normal.

\item[(ii)] There exists a Riemannian manifold $N$ such that $M$ is isometric to a submanifold of $N$ and
each smooth subsurface $S$ of $\partial M$  extends to a smooth embedded surface $S'$ in $N$ such that $S' \cap M = S$.

\end{itemize}

If in (ii) of the definition we require that each smooth subsurface has strictly positive mean curvature, then $M$ is called {\em strictly mean convex}.
\end{defn}

\begin{defn} A simple closed curve is called an {\em extreme curve} if it is on the boundary of its convex hull.
A simple closed curve is called an {\em $H$-extreme curve} if it is a curve in the boundary of a mean convex manifold $M$.
\end{defn}

\begin{defn} An embedded surface $S$ in a $3$-manifold $M$ is called {\em properly embedded} if $S\cap \partial M = \partial S$.
\end{defn}

We now state the main facts which we will be using in the following sections.

\begin{thm}\label{meeksyau}\cite{MY2, MY3}
Let $M$ be a compact, mean convex $3$-manifold, and $\Gamma\subset\partial M$ be a nullhomotopic simple closed curve. Then, there exists an area
minimizing disk $ D\subset M$ with $\partial  D = \Gamma$. Moreover, all such disks are properly embedded in $M$ and they are pairwise disjoint.
Furthermore, if $\Gamma_1, \Gamma_2 \subset \partial M$ are two disjoint simple closed curves, then two area minimizing disks $ D_1$ and $ D_2$ spanning $\Gamma_1$ and $
\Gamma_2$ respectively are also disjoint.
\end{thm}

There is an analogous fact for area minimizing surfaces, too.

\begin{thm}\cite{FeF, ASSi, H}\label{absmincur}
Let $M$ be a compact, strictly mean convex $3$-manifold and $\Gamma\subset\partial M$ a nullhomologous simple closed curve. Then there exists $\Sigma\subset M$ an
absolutely area minimizing surface with $\partial \Sigma = \Gamma$ and each such $\Sigma$ is smooth away from its
boundary and it is smooth around points of the boundary where $\Gamma$ is smooth.

\end{thm}

\begin{rmk}[{\it on the proof of Theorem \ref{absmincur}}] The regularity results in \cite{ASSi, H} extend from the ambient space being $\R^n$ to $M$, a manifold
as in the Theorem \ref{absmincur}, as explained in \cite[Theorem 6.3]{Wh2}.
 \end{rmk}

\begin{thm}\label{white2} \cite{Wh2} Let $M$ be a compact, orientable, strictly mean convex $3$-manifold and $\Gamma$ a simple closed curve in $\partial M$.
For each $g\geq 0$, define \[a(g)=\inf|\Sigma|\] where the infimum is taken over all piecewise smooth embedded surfaces $\Sigma$ in $M$, with genus
$g$ and boundary $\Gamma$ and where $|\Sigma|$ denotes the area of a surface $\Sigma$. Then, if $a(g)< a(g-1)$, the infimum $a(g)$ is attained.
\end{thm}

\begin{thm}\label{white2c}\cite{Wh2}
Let $M$ be a compact, orientable, strictly mean convex $3$-manifold, $\Gamma$ a simple closed curve in $\partial M$ and $\Sigma$ a stable surface in
$M$ with $\partial \Sigma=\Gamma$. Assume that for a geodesic ball $\B(x, R)$ in $M$ we have
\[\sup_{\B(y,r)\subset\B(x,R)}\frac{|\Sigma\cap\B(y,r)|}{\pi r^2}=C<\infty.\]
Then the following hold:
\begin{itemize}
\item If $\B(x,R)\cap\Gamma=\emptyset$, then the principal curvatures of $M\cap \B(x, R/2)$ are bounded by a constant that depends only on $R, C$.
\item  If $\B(x,R)\cap\Gamma$ is $C^{2,\a}$ , then the principal curvatures of $M\cap \B(x, R/2)$ are bounded by a constant that depends only on $R, C$ and $\Gamma$.
\end{itemize}
\end{thm}

Finally, we state two results on the generic uniqueness of area minimizing disks, and absolutely area minimizing surfaces for $H$-extreme curves.

\begin{thm}\cite{Co1}
Let $M$ be a compact, orientable, mean convex $3$-manifold with $H_2(M,\BZ)=0$. Then the following hold:
\begin{itemize}
\item[(i)]
For a generic nullhomotopic (in
$M$) simple closed curve $\Gamma$ in $\partial M$, there exists a unique area minimizing disk $D$ in $M$ with $\partial D = \Gamma$.
\item[(ii)] For a generic nullhomologous (in
$M$) simple closed curve $\Gamma$ in $\partial M$, there exists a unique absolutely area minimizing surface $\Sigma$ in $M$ with $\partial \Sigma =
\Gamma$.

\end{itemize}
\end{thm}

\noindent {\bf Convention:} Throughout the paper, all the manifolds will be assumed to be compact, orientable, and strictly mean convex unless
otherwise stated. We will also assume that all the surfaces are orientable.

\section{Stratification of the space of simple closed curves in $\partial M$}

In this section, we study the space of nullhomologous simple closed curves in $\partial M$, where $M$ is a compact, orientable, strictly mean convex $3$-manifold. We
stratify this space with respect to the minimum genus of the absolutely area minimizing surfaces that its curves bound.

 Let $\A$ denote the space of nullhomologous simple closed curves in $\partial
M$, equipped with the $\mathcal{C}^0$ topology.  We give here the precise definition of the neighborhoods defining this topology.
\begin{defn}\label{c0rmk}
Let $\Gamma\in\A$, then there exists a continuous function
\[f: S^1\to\partial M\]
such that $f(S^1)=\Gamma$, where $S^1=\partial \BB^2$ and $\BB^2$ is the unit 2-ball in $\R^2$.

An $\e$-neighborhood of $\Gamma$, $\U_\Gamma(\e)$ in the $\mathcal{C}^0$ topology consists of all $\wt\Gamma\in\A$ for which there exists a continuous function
\[g: S^1\to\partial M\]
such that $g(S^1)=\wt\Gamma$ and $\displaystyle\sup_{x\in S^1}|f(x)-g(x)|<\e$.

Note that, for $\e$ small enough, so that the $\e$-neighborhood of $\Gamma$ in $\partial M$, defined by
\[N_\e(\Gamma)=\{x\in\partial M:\dist(x, \Gamma)<\e\},\]
is topologically an annulus, we have that  $\wt\Gamma\in\mathcal \U_\Gamma(\e)$ is equivalent to  $\wt\Gamma\subset N_\e(\Gamma)$ and
$\wt\Gamma\sim\Gamma$, i.e. $\wt\Gamma$ is homotopic to $\Gamma$.
\end{defn}

We define a relation $\varphi: \A \to \BN$, by
$$\varphi(\Gamma) = \min_\Sigma \genus(\Sigma),$$
where the minimum is taken over all  absolutely area minimizing surfaces $\Sigma$ in  $M$ with $\partial \Sigma = \Gamma$ and naturally
$\genus(\Sigma)$ denotes the genus of a surface $\Sigma$.

Now, define $\A_g = \{ \ \Gamma\in \A \ | \ \varphi(\Gamma)\geq g \ \}$ and $\B_g = \varphi^{-1}(g)= \A_g \setminus\A_{g+1}$. In particular, $\A_g$
denotes the space of simple closed curves in $\partial M$ such that the minimum genus of the absolutely area minimizing surfaces that they bound is
greater than or equal to $g$. $\B_g$ denotes the space of simple closed curves in $\partial M$ such that the minimum genus of the absolutely area
minimizing surfaces that they bound is exactly $g$. Clearly, $\A=\A_0\supset \A_1\supset \A_2 \supset ...\supset \A_n \supset ...$.

\begin{lem}\label{Agopen}
For any $g\geq 0$, $\A_g$ is open in $\A$.
\end{lem}

\begin{pf} Let  $\Gamma_0\in \A_g$ be a simple closed curve in $\partial M$. Since $\Gamma_0\in
\A_g$, we have $\varphi(\Gamma_0)\geq g$, i.e. if $\Sigma$ is an absolutely area minimizing surface in $M$ with $\partial \Sigma = \Gamma_0$, then
$\genus(\Sigma) \geq g$.

Assume now that the Lemma is not true. Then for any $\e>0$, there exists $\Gamma_\e\in\U_{\Gamma_0}(\e)$, such that $\Gamma_\e\notin\A_g$, where
$\U_{\Gamma_0}(\e)$ is the $\e$-neighborhood of $\Gamma_0$ in the $\C^0$ topology, as in Definition \ref{c0rmk}. Note that for $\e$ small enough,
$\Gamma_\e\sim\Gamma$ (cf. Definition \ref{c0rmk}). Hence, for $\e$ small enough, there exists a region $A_\e$  in $\partial M$ between $\Gamma_\e$
and $\Gamma$, i.e. $\partial A_\e=\Gamma_\e\cup\Gamma$, for which we have $|A_\e|\stackrel{\e\to 0}{\longrightarrow} 0$, where $|\cdot|$ denotes the
area of a surface. This implies that for any absolutely area minimizing surface $\Sigma_\e$, with $\partial\Sigma_\e=\Gamma_\e$, we have
$|\Sigma_\e|\le |\Sigma|+|A_\e|$, where $\Sigma$ is  any absolutely area minimizing surface with boundary equal to $\Gamma_0$. Therefore there exists
a sequence $\{\Gamma_i\}_{i\in\N}\subset\A$, such that $\Gamma_i\stackrel{i\to\infty}{\longrightarrow}\Gamma_0$ in the $\C^0$ topology and such that
for each $i$, there exists an absolutely area minimizing surface $\Sigma_i$, with $\partial\Sigma_i=\Gamma_i$, $\genus(\Sigma_i)\le g-1$ and such
that their areas satisfy $|\Sigma_i|\stackrel{i\to\infty}{\longrightarrow}|\Sigma|$.

Let now $S_i=[\![\Sigma_i]\!]$, be the corresponding integral currents. Then, since the areas $|\Sigma_i|$ are uniformly bounded, by the
Federer-Fleming compactness theorem for integral currents \cite{FeF}, after passing to a subsequence $S_i\to S_0$, in the sense of currents, where
$S_0$ is an integral current, with $\mass(S_0)\le \lim_{i\to\infty} \mass(S_i)=|\Sigma|$; here $\mass$ denotes the mass of the current (see
\cite{Sim}). Hence $\mass(S_0)=|\Sigma_0|$ and thus (using Theorem \ref{absmincur}) $S_0=[\![\Sigma_0]\!]$, where $\Sigma_0$ is an absolutely area
minimizing surface with $\partial \Sigma_0=\Gamma_0$ and since $\Gamma_0\in\A_g$, we have that $\genus(\Sigma_0)\ge g$.

Note now that, since $S_i$ are  absolutely area minimizing, in the above convergence, the corresponding radon measures converge $\mu_{S_i}\to \mu_S$
(for a proof of this see \cite[Theorem 34.5]{Sim}). Furthermore, since each $\Sigma_i$ is a minimal surface, the monotonicity formula holds
\cite{Sim}, i.e. for each $i$, each $x\in \ov M\setminus \Gamma_i$ and for $r$ small enough (so that $\B(x, r)\cap\Gamma_i=\emptyset$),
$r^{-2}|\Sigma_i|$ is an increasing function of $r$. Therefore we can apply the principal curvature bound of White, as given in Theorem \ref{white2c},
to conclude that the convergence $\Sigma_i\to\Sigma_0$ is actually smooth in compact sets of $\overline M\setminus\Gamma_0$. Let now $K$ be a compact
subset of $M$, such that $\Sigma_0\cap K$ has genus $\ge g$. Then the previous convergence, along with the Gauss-Bonnet theorem, implies that the
same should be true for the $\Sigma_i$'s, for $i$ big enough, i.e. that $\genus(\Sigma_i\cap K)\ge g$. But this contradicts the fact that $\genus(
\Sigma_i)\le g-1$.\end{pf}

\subsection{A bridge principle for Absolutely Area Minimizing Surfaces} \label{bridgeprinciplesubsection}
\

\

In this section we define a surgery operation on a simple closed curve $\Gamma$ in $\partial M$, by gluing a ``bridge'' on $\Gamma$. Let $\alpha$ be
a simple path in $\partial M$ such that $\a$ intersects $\Gamma$ transversely and $\Gamma \cap \alpha = \{x,y\}$ are the endpoints of $\alpha$. Let
$\epsilon_0 >0$ be sufficiently small so that for almost every $0<\epsilon<\epsilon_0$, $\partial N_{\epsilon}(\alpha)\cap \Gamma$ consists of
exactly 4 points, namely $\{x^+_{\epsilon}, x^-_{\epsilon}, y^+_{\epsilon}, y^-_{\epsilon}\}$, where $N_{\epsilon}(\alpha)$ is the $\epsilon$
neighborhood of $\alpha$ in $\partial M$, i.e.
\[N_\e(\alpha)=\{x\in\partial M:\dist(x,\alpha)<\e\}.\]
Then, $\Gamma$ divides $N_{\epsilon}(\alpha)$ into three components, which we call $C_x, C_y$ and $S^\epsilon_\alpha$, where $S^\epsilon_\alpha$ is
the component containing $\alpha$ and $C_x, C_y$ are the remaining caps of $N_{\epsilon}(\alpha)$ near $x$ and $y$ respectively (see Figure 1).

For orientation issues, that are explained in Remark \ref{OI}, we further assume that the path $\alpha$ is such that $C_x$ and $C_y$ are in the same
side of $\Gamma$. In particular we will be considering paths $\a$ satisfying the following definition.

\begin{defn}\label{a}
 Let $\Gamma\in\A$ and let $\alpha$ be a simple path in $\partial M$ such that $\a$ intersects $\Gamma$ transversely and $\Gamma
\cap \alpha = \{x,y\}$ are the endpoints of $\alpha$.
We say that $\a$ is a {\em
$\Gamma$-admissible path}  if the following holds:

 If $N_{\epsilon '} (\Gamma)$ is a small annular neighborhood of $\Gamma$ in $\partial M$, and $N_{\epsilon '} (\Gamma) \setminus \Gamma =
H^+\cup H^-$ where $H^+$ is an annulus, say the ``positive side" of $\Gamma$, and $H^-$ is the other annulus, say the ``negative side" of $\Gamma$,
then  either both $C_x$ and $C_y$ intersect $H^+$ or they both intersect $H^-$, where $C_x, C_y$ are as above (see Figure 1).
\end{defn}

\begin{figure}[h]

\relabelbox  {\epsfxsize=5in

\centerline{\epsfbox{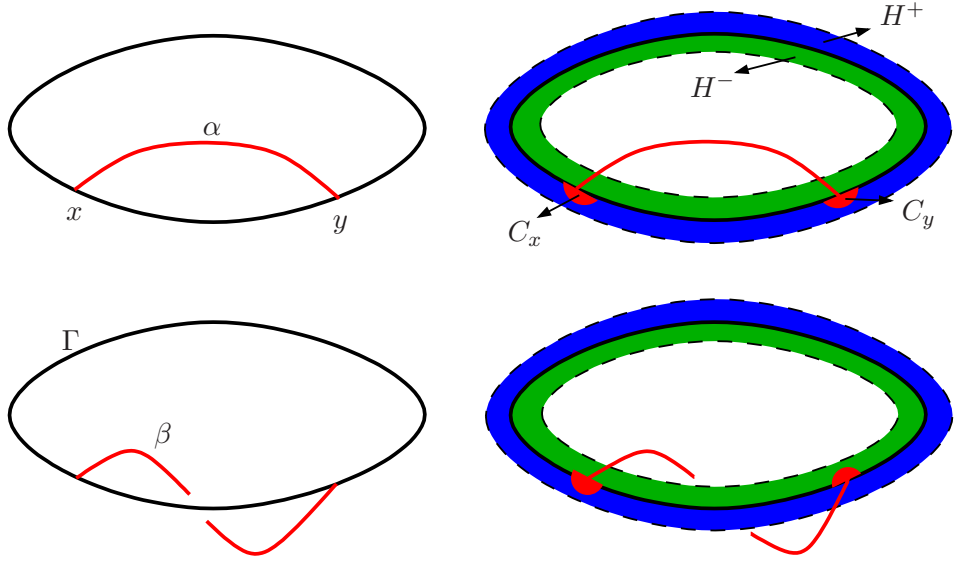}}}

\relabel{1}{$\Gamma$}

\relabel{2}{$x$}

\relabel{3}{$y$}

\relabel{4}{$H^+$}

\relabel{5}{$H^-$}

\relabel{6}{$C_x$}

\relabel{7}{$C_y$}

\relabel{8}{$\beta$}

\relabel{9}{$\alpha$}

\endrelabelbox

\caption{\label{fig:figure2} \small {$\alpha$ is $\Gamma$-admissible as both $C_x$ and $C_y$ belongs to $H^+$. $\beta$ is not $\Gamma$-admissible as
$C_x$ belongs to $H^+$ while $C_y$ belongs to $H^-$. }}

\end{figure}

$S^\epsilon_\alpha$, as constructed above is a tiny strip (``bridge'') of width $\e$, around $\alpha$. For its boundary we have $\partial
S^\epsilon_\alpha = (\alpha^+_\epsilon \cup \alpha^-_\epsilon) \cup (\beta^x_\epsilon \cup \beta^y_\epsilon)$ where $\beta^x_\epsilon$ is the arc
segment of $\Gamma$ between $x^+_\epsilon$ and $x^-_\epsilon$ containing $x$, and similarly $\beta^y_\epsilon$ is the arc segment of $\Gamma$ between
$y^+_\epsilon$ and $y^-_\epsilon$ containing $y$. Also, $\alpha^+_\epsilon$ is a simple path in $\partial M$ connecting $x^+_\epsilon$ to
$y^+_\epsilon$, and $\alpha^-_\epsilon$ is a simple path in $\partial M$ connecting $x^-_\epsilon$ to $y^-_\epsilon$ in $\partial M$ (See Figure 2).
``Gluing'' this strip $S^\e_\a$ on $\Gamma$ results to the following definition.

\begin{defn}\label{wtG}
Let $\Gamma\in\A$ and let $\a$ be a $\Gamma$- admissible path. Then we define a new curve
\[\widehat{\Gamma}^\epsilon_\alpha=\Gamma\sharp S_\a^\e,\]
obtained by gluing on $\Gamma$ a bridge $S^\e_\a$ of width $\e$ around $\a$.
With the above notation $\Gamma\sharp S_\a^\e$ is the curve
\[\widehat{\Gamma}^\epsilon_\alpha =\Gamma\sharp S_\a^\e=(\Gamma \setminus (\beta^x_\epsilon \cup \beta^y_\epsilon)) \cup (\alpha^+_\epsilon \cup \alpha^-_\epsilon).\]
Since $\alpha$ is $\Gamma$-admissible, $\widehat{\Gamma}^\epsilon_\alpha$ is the union of two simple closed curves in $\partial M$, which we will denote by
$\widehat{\Gamma}^{\epsilon,1}_\alpha$ and $\widehat{\Gamma}^{\epsilon,2}_\alpha$ (See Figure 2).

\end{defn}



\begin{figure}[h]

\relabelbox  {\epsfxsize=5in

\centerline{\epsfbox{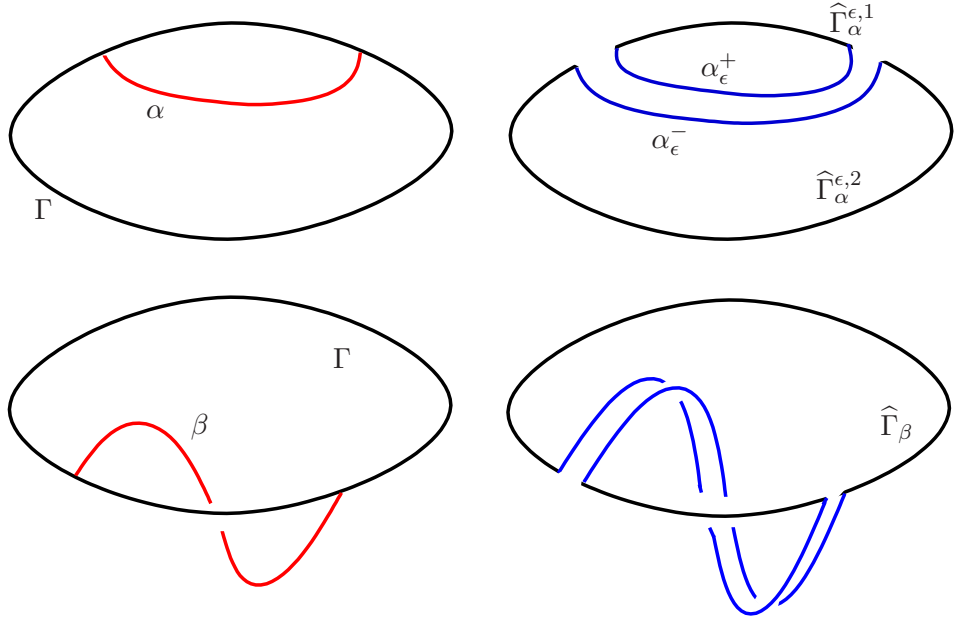}}}

\relabel{1}{$\Gamma$}

\relabel{2}{$\alpha$}

\relabel{3}{$\widehat{\Gamma}^{\epsilon,2}_\alpha$}

\relabel{4}{$\alpha_\epsilon^-$}

\relabel{5}{$\alpha_\epsilon^+$}

\relabel{6}{$\widehat{\Gamma}^{\epsilon,1}_\alpha$}

\relabel{7}{$\Gamma$}

\relabel{8}{$\beta$}

\relabel{9}{$\widehat{\Gamma}_\beta$}

\endrelabelbox

\caption{\label{fig:figure2} \small {If $\alpha$ is $\Gamma$-admissible, after the surgery along $S^\epsilon_\alpha$, the new curve
$\widehat{\Gamma}^\epsilon_\alpha$ has two components, $\widehat{\Gamma}^{\epsilon,1}_\alpha$ and $\widehat{\Gamma}^{\epsilon,2}_\alpha$. When
$\beta$ is not $\Gamma$-admissible, then the resulting curve $\widehat{\Gamma}^\epsilon_\beta$ has only one component.}}

\end{figure}

The above gluing construction is also possible when instead of a single simple closed curve, we have two disjoint curves and $\a$ is a path joining
them. In particular, let $\Gamma_1, \Gamma_2\in \A$ be such that $\Gamma_1\cap\Gamma_2=\emptyset$. Then the above gluing operation goes through in
exactly the same way as described, with $\Gamma$ replaced by $\Gamma_1\cup\Gamma_2$. In this case we have no orientation issues but we want $\a$ to
be a path ``connecting'' the two curves. In particular we will be considering paths $\a$ satisfying the following definition.

\begin{defn}\label{a2}
 Let $\Gamma_1, \Gamma_2\in \A$ be such that $\Gamma_1\cap\Gamma_2=\emptyset$, and let $\alpha$ be a simple path in $\partial M$ such that $\a$ intersects
 $\Gamma_1$ and $\Gamma_2$ transversely and $(\Gamma_1\cup\Gamma_2) \cap \alpha = \{x,y\}$ are the endpoints of $\alpha$. We say that $\a$ is a path {\em
 connecting } $\Gamma_1$ and $\Gamma_2$  if $x\in\Gamma_1$ and $y\in \Gamma_2$.
\end{defn}

As in the case of a simple closed curve we can consider $S_\a^\e$, a tiny strip (``bridge'') of width $\e$ around $\a$ with $\partial
S^\e_\a=(\a^+_\e\cup\a^-_\e) \cup(\beta_\e^x\cup\beta_\e^y)$, where $S_\a^\e, a^+_\e, \a^-_\e, \beta_\e^x$ and $\beta_\e^y$ are defined exactly as
before. Connecting $\Gamma_1 $ and $\Gamma_2$ via this stip $S^\e_\a$  results to the following definition.

\begin{defn}\label{wtG2}
Let $\Gamma_1, \Gamma_2\in\A$ be such that $\Gamma_1\cap\Gamma_2=\emptyset$ and let $\a$ be a path connecting them. Then we define a new curve
\[\widehat{\Gamma}^\epsilon_\alpha=\Gamma\sharp S_\a^\e,\]
obtained by gluing on $\Gamma$ a bridge $S_\a^\e$ of width $\e$ around $\a$,
where $\Gamma=\Gamma_1\cup\Gamma_2$.
With the previous notation $\Gamma\sharp S_\a^\e$ is the curve
\[\widehat{\Gamma}^\epsilon_\alpha =\Gamma\sharp S_\a^\e=(\Gamma \setminus (\beta^x_\epsilon \cup \beta^y_\epsilon)) \cup (\alpha^+_\epsilon \cup \alpha^-_\epsilon).\]
Since $\alpha$ is a path connecting $\Gamma_1$ and $\Gamma_2$, $\widehat{\Gamma}^\epsilon_\alpha$ is a simple closed curve.
\end{defn}

\begin{rmk}\label{smoothglue}
We remark here that if the curve $\Gamma$ is smooth around $\Gamma\cap\a$, where $\a$ is a $\Gamma$-admissible path, then we can consider gluing the
bridge $S^\e_\a$ smoothly along $\Gamma$, so that the resulting curve $\widehat{\Gamma}_\a^\e=\widehat{\Gamma}_\a^{\e,1}\cup
\widehat{\Gamma}_\a^{\e,2}$, as in Definition \ref{wtG}, is smooth around $x^\pm, y^\pm$.

Similarly, if $\Gamma_1$ and  $\Gamma_2$ are smooth around $(\Gamma_1\cup\Gamma_2)\cap\a$, where $\a$ is a path connecting $\Gamma_1$ and $\Gamma_2$,
then we  consider gluing the bridge $S^\e_\a$ smoothly along $\Gamma_1$ and $\Gamma_2$, so that the resulting curve $\widehat{\Gamma}_\a^\e$, as in
Definition \ref{wtG2}, is smooth around $x^\pm, y^\pm$.

In the rest of the paper and as long as these smoothness assumptions are satisfied we will always consider constructing the curve
$\widehat{\Gamma}_\a^\e$ to be smooth around $\a$.

\end{rmk}

\begin{lem} {\bf (A bridge principle for absolutely area minimizing surfaces)}\label{bridge principle}

\noindent Assume that either
\begin{itemize}
\item[(a)] $\Gamma\in\A$ and $\alpha$ is a $\Gamma$-admissible simple path in $\partial M$ (as in \emph{Definition \ref{a}}) or
\item[(b)] $\Gamma=\Gamma_1\cup\Gamma_2$, where $\Gamma_1,\Gamma_2\in\A$ are such that $\Gamma_1\cap\Gamma_2=\emptyset$ and $\alpha$ is a simple path in
$\partial M$ connecting them (as in \emph{Definition \ref{a2}}).
\end{itemize}

Let $\widehat{\Gamma}^\epsilon_\alpha=\Gamma\sharp S^\e_\a$ be the curve obtained by attaching on $\Gamma$ a strip of width $\e$ around $\a$  (see \emph{Definition \ref{wtG}} and
\emph{\ref{wtG2}} for cases \emph{(a)} and \emph{(b)} respectively). Then, for any $\eta>0$ there exists $\epsilon_\alpha>0$ such that for almost every
$0<\epsilon<\epsilon_\alpha$, any absolutely area minimizing surface in $M$, $\Sigma^\e_\alpha$, with $\partial
\Sigma^\e_\alpha=\widehat{\Gamma}^\epsilon_\alpha$ is $\eta-$close in Hausdorff distance and in area, to the surface obtained by attaching an
$\e$-strip around $\alpha$ to one of the absolutely area minimizing surface bounded by $\Gamma$, i.e.
\[\sup_{x\in \Sigma^\e_\alpha}\dist(x, \Sigma\cup S^\e_\alpha)<\eta\]
and
\[\left||\Sigma_\a^\e|-|\Sigma|\right|<\eta\]
for some absolutely area minimizing surfaces $\Sigma$, with $\partial \Sigma=\Gamma$.

\end{lem}

\begin{pf}
We prove this by contradiction. Assume that the lemma is not true, then there exist $\Gamma, \a$ satisfying either the assumption (a) or the
assumption (b) of the lemma, and  some $\eta>0$, such that the following holds: For any $\e_\a>0$, there exists $\e<\e_\a$ and an absolutely area minimizing
surface $\Sigma^\e_\alpha$ in $M$, with $\partial \Sigma^\e_\alpha=\widehat{\Gamma}^\epsilon_\alpha$, such that for any absolutely area minimizing
surface $\Sigma$, with $\partial \Sigma=\Gamma$ either
\begin{equation}\label{supcontr}
\sup_{x\in \Sigma^\e_\alpha}\dist(x, \Sigma\cup S^\e_\alpha)\ge\eta
\end{equation}
or
\begin{equation}\label{supcontr2}
\left||\Sigma_\a^\e|-|\Sigma|\right|\ge\eta
\end{equation}
Hence there exists a sequence $\e_i\downarrow 0$, and surfaces
$\Sigma_\alpha^{\e_i}$ as above such that for any absolutely area minimizing surface $\Sigma$,
with $\partial \Sigma=\Gamma$,  either \eqref{supcontr} or \eqref{supcontr2} holds, with $\e$ replaced by $\e_i$.

Let $T_i=[\![\Sigma_\a^{\e_i}]\!]$ be the corresponding currents. Note that for any absolutely area minimizing current $T_0$, with $\partial T_0=[\![\Gamma]\!]$, we have that
\[\partial T_i=\partial(T_0+[\![S^{\e_i}_\a]\!])\]
and thus
\[\underline{\underline M}(T_i)\le \underline{\underline M}(T_0+[\![S^{\e_i}_\a]\!])\le \underline{\underline M}(T_0)+C\e_i,\]
where $\underline{\underline M}$ denotes the mass and $C$ is a constant that depends on $\a$ (independent of $i$). Hence, the currents $T_i$ have
uniformly bounded masses, and thus by the Federer-Fleming compactness theorem \cite{FeF}, after passing to a subsequence, $T_i\to T$, in the sense of
currents, where $T$ is an integral current. Furthermore (cf. \cite[Theorem 34.5]{Sim}) $T$ is absolutely area minimizing, $\partial T=\Gamma$ and the
corresponding Radon measures also converge $\mu_{T_i}\to \mu_T$. Therefore, by Theorem \ref{absmincur}, $T=[\![\Sigma_0]\!]$, where $\Sigma_0$ is an absolutely area minimizing surface and it is also
smooth and embedded away from $\Gamma$. The measure convergence implies that

\[\left||\Sigma_\a^{\e_i}|-|\Sigma_0|\right|\stackrel{i\to\infty}{\longrightarrow}0\]
and so \eqref{supcontr2} cannot hold. Therefore we assume that for every $i$, \eqref{supcontr} holds.
 In particular, for each $i$, there exists $x_i\in \Sigma_\a^{\e_i}$ such that
\[\dist(x_i, \Sigma\cup S^{\e_i}_\alpha)\ge\eta\]
for any absolutely area minimizing surface $\Sigma$, with $\partial \Sigma=\Gamma$.

By passing to a further subsequence if necessary we also have that
\[x_i\to x_0\]
for some $x_0\in \ov M$. Note that by the assumption on $x_i$, we have that $\dist(x_i,\Sigma_0\cup\alpha)\ge \eta$ and hence we also have that $\dist(x_0,
\Sigma_0\cup\alpha)\ge \eta$.

By the above convergence, there exists $i_0$, such that $\forall i\ge i_0$ we have
\[\B(x_i,{\eta/4})\subset \B(x_0,{\eta/2})\subset \B(x_i, \eta),\]
where $\B(x,\rho)$ denotes a geodesic ball of radius $\rho$ and centered at $x$, in $M$.
Hence, using the monotonicity formula (cf. \cite{Sim}) and the fact that $x_i\in\spt T_i\setminus \spt\partial T_i$
\[\mu_{T_i}(\B(x_0, {\eta/2}))\ge\mu_{T_i}(\B(x_i,{\eta/4}))\ge \pi\left(\frac\eta 4\right)^2.\]
Using now the measure convergence, we have that
\[\mu_T(\B(x_0,{\eta/2}))=\lim_i\mu_{T_i}(\B(x_0,{\eta/2}))\ge \pi\left(\frac\eta 4\right)^2\]
and therefore
\[\mu_T(\B(x_i,{\eta}))\ge\mu_T(\B(x_0,{\eta/2}))\ge  \pi\left(\frac\eta 4\right)^2\]
which implies that
\[\spt T\cap \B(x_i,\eta)\ne \emptyset\Rightarrow \dist (x_i, \Sigma_0\cup\alpha)<\eta.\]
This contradicts the choice of $x_i$ and thus \eqref{supcontr} cannot hold for every $i$. This finishes the proof of the lemma.
\end{pf}

\begin{rmk}\label{C2convrmk} Let $\{\e_i\}_{i\in\N}$, be a sequence with $\e_i\downarrow 0$ and for each $i$ let $\Sigma_\a^{\e_i}$ be an area minimizing
surface with $\partial \Sigma_\a^{\e_i}=\widehat\Gamma_{\a}^{\e_i}$ as in Theorem \ref{bridge principle}. Then the proof of Theorem \ref{bridge
principle} shows that after passing to a subsequence, $\Sigma_\a^{\e_i}\to\Sigma$, where $\Sigma$ is an absolutely area minimizing surface with
$\partial\Sigma=\Gamma$ and the convergence is with respect to Hausdorff distance and with respect to measure. Therefore we can argue as in the proof
of Lemma \ref{Agopen}, using White's curvature bound, Theorem \ref{white2c}  (which is applicable because of the measure convergence and the monotonicity formula
\cite{Sim}), to conclude that this convergence is smooth in compact sets of $\ov M\setminus (\Gamma\cup\a)$ and in fact it is also smooth around
points of $\Gamma\setminus\a$, where $\Gamma$ is smooth. In particular,  if $\Gamma$ is smooth, then  it is a smooth convergence in compact sets of
$\ov M\setminus \a$.
\end{rmk}

\begin{rmk} Note that the result of Lemma \ref{bridge principle} is not an extension of the classical bridge principle to absolutely area minimizing surfaces.
In the classical bridge principle, one starts with a fixed stable minimal surface $T$ with boundary $\Gamma$, and
after adding a thin bridge $S^\epsilon_\alpha$ to $\Gamma$, one gets a new surface $T \sharp S^\epsilon_\alpha$ with boundary
$\widehat{\Gamma}^\epsilon_\alpha$ which is close to the original surface $T$ \cite{MY3}, \cite{Wh4}.

In our case, if $\Gamma$ bounds more than one absolutely area minimizing surface, say $\{\Sigma_1, \Sigma_2,..., \Sigma_k\}$, then the absolutely
area minimizing surface $\Sigma^\epsilon_\alpha$ bounded by $\widehat{\Gamma}^\epsilon_\alpha$ is close to $\Sigma_{i_0} \cup S^\epsilon_\alpha$ for
some $1\leq i_0 \leq k$. Therefore, one may not get a modified absolutely area minimizing surface $\Sigma_\a^\e$ which is close to $\Sigma_i$ for
each $i$. The absolutely area minimizing surface $\Sigma_{i_0}$ which is close to $\Sigma^\epsilon_\alpha$ depends on the choice of the simple path
$\alpha$ and then $\Sigma_{i_0}$ is the limit of the sequence of absolutely area minimizing surfaces $\{\Sigma^{\epsilon_i}_\alpha\}$ in $M$ as it
appears in the proof of Theorem \ref{bridge principle}. At this point, we conjecture that $\Sigma_{i_0}$ must be one of the two canonical absolutely
area minimizing surfaces $\Sigma^+,\Sigma^-$ bounding $\Gamma$ as in Lemma 4.2 of \cite{Co1}, since they are the extremal absolutely area minimizing
surfaces with boundary $\Gamma$.

On the other hand, if one relax the condition of area minimizing to just minimal for the new surface, then the classical bridge principle for extreme
curves is valid  by using the techniques in Section 4. There, we show how one can apply \cite{Wh4} (see Theorem \ref{white4} below) to get a {\em
minimal surface} $S$ with $\partial S = \widehat{\Gamma}^\epsilon_\alpha$, which is close to $\Sigma_{i_0} \cup S^\epsilon_\alpha$ for a specified area
minimizing surface $\Sigma_{i_0}$ with $\partial \Sigma_{i_0} = \Gamma$.
\end{rmk}

\begin{rmk}\label{OI} (Orientation Issues) If $\Sigma$ is an oriented surface whose boundary is a simple closed curve $\Gamma$, gluing a strip $S\sim I\times
I$ ``trivially" along the boundary $\Gamma$ will give another oriented surface whose boundary consists of $2$ simple closed curves. However, if the
strip $S$ glued to $\Sigma$ along $\Gamma$ has a ``twist", then the new surface will not be oriented anymore, and the boundary will be a simple
closed curve again. The $\Gamma$-admissibility condition on $\alpha$ expresses this difference. Therefore, the surgery along a $\Gamma$-admissible
path $\alpha$ gives an oriented surface with $2$ boundary components, whereas the surgery along a non-$\Gamma$-admissible path $\beta$ would be
gluing a twisted strip along $\Gamma$ to $\Sigma$, and give a nonorientable surface whose boundary is a simple closed curve again (See Figure 3).
\end{rmk}

\begin{figure}[h]

\relabelbox  {\epsfxsize=5in

\centerline{\epsfbox{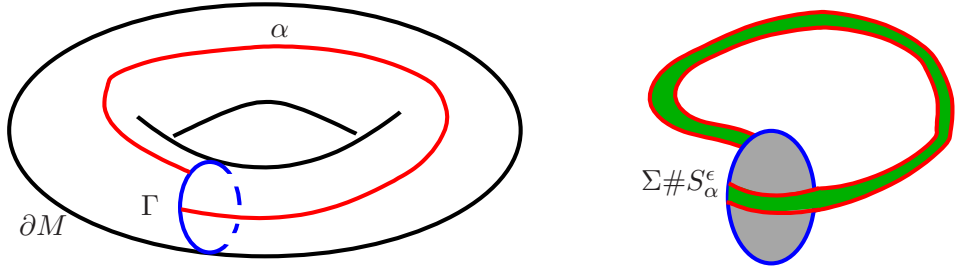}}}

\relabel{1}{$\partial M$}

\relabel{2}{$\Gamma$}

\relabel{3}{$\alpha$}

\relabel{4}{$\Sigma \# S^\epsilon_\alpha$}

\endrelabelbox

\caption{\label{fig:figure2} \small {If $\alpha$ is not $\Gamma$-admissible, when we attach the strip $S^\epsilon_\alpha$ to the oriented surface
$\Sigma$ in $M$ with boundary $\Gamma$, we get a nonorientable surface. In this picture, $\Sigma \# S^\epsilon_\alpha$ is a Mobius band. }}

\end{figure}

Next we show in Corollary~\ref{genuscor} how one can relate the genus of the surfaces $\Sigma^\e_\a$ to that of $\Sigma$, where $\Sigma^\e_\a,
\Sigma$ are as in Lemma \ref{bridge principle}. For the proof, we need White's bridge principle for stable minimal surfaces \cite{Wh4}. For the
convenience of the reader, we state it here in the form that we will apply it, using our notation so far.

\begin{thm}\label{white4}\cite{Wh4}
Let $\Gamma, \a$ satisfy assumption \emph{(a)} or \emph{(b)} of \emph{Theorem \ref{bridge principle}}. We further assume that both $\Gamma$ and $\a$ are
smooth. Let $\Sigma$ be a smooth and strictly stable surface with $\partial\Sigma=\Gamma$. Let also $\Sigma_n$ be a sequence of minimal surfaces
converging smoothly to $\Sigma$ and such that for $\Gamma_n:=\partial\Sigma_n$, we have that $W\cap \Gamma_n= W\cap\Gamma$, for an open set $W$
containing $\a$. Then:
\begin{itemize}
\item[i.] There is an open set  $U$ containing $\Sigma\cup\a$ such that $\Sigma$ is the unique area minimizing surface in $\ov{U}$.
\item[ii.] If for a sequence $\e_n\downarrow 0$, $T_n$ minimizes area among all surfaces in $\ov{U}$ with boundary ${\widehat\Gamma^{\e_n}_{n\a}}$,
where ${\widehat\Gamma^{\e_n}_{n\a}}$ are the smooth curves given in \emph{Definition \ref{wtG}} or \emph{\ref{wtG2}} (see also \emph{Remark \ref{smoothglue}}),
using now $\Gamma_n$ instead of $\Gamma$, then
\begin{itemize}
\item$|T_n|\stackrel{n\to\infty}{\longrightarrow}|\Sigma|$.
\item $T_n$ converges smoothly to $\Sigma$ on compact subsets of $\ov M\setminus \a$.
\item For sufficiently large $n$, $T_n$ and $\Sigma_n\cup S_\a^\e$ are diffeomorphic.
\item  For sufficiently large $n$, $T_n$ is unique and strictly stable.
\end{itemize}
\end{itemize}
\end{thm}

\begin{rmk}\label{rmkonwhite4}
In \cite{Wh4}, part (i) of the above theorem is actually proven without any assumptions on the regularity of $\Gamma$ and furthermore the Theorem
holds, without necessarily $\Gamma$ lying  in the boundary of $M$. Also, in \cite{Wh4} it is actually proven (and we will need it later) that Theorem \ref{white4}
 still holds if $\Sigma_0$ is an immersed surface. Note also that although $\Gamma$  might not lie in $\partial M$, as long as $\Gamma\cap
W\subset\partial M$ and thus also $\Gamma_n\cap W\subset\partial M$, the construction of the curves ${\widehat\Gamma^{\e_n}_{n\a}}$ given in
Definitions \ref{wtG} and \ref{wtG2} still makes sense.
\end{rmk}

\begin{lem}\label{genuslem}
Let $\Gamma, \a$ satisfy assumption \emph{(a)} or \emph{(b)} of \emph{Theorem \ref{bridge principle}}. We further assume that $\a$ is smooth,
$\Gamma$ is  smooth in an open set $W$ containing $\a$ and for any absolutely area minimizing surface $\Sigma$ with $\partial\Sigma=\Gamma$,
$\genus(\Sigma)<\infty$.  Then there exists $\e_\a>0$ such that for almost every $0<\e<\e_a$ and for any absolutely area minimizing surface
$\Sigma_\a^{\e}$, with $\partial\Sigma_\a^{\e}=\widehat\Gamma_\a^\e$, where $\widehat\Gamma_\a^\e$ is as in \emph{Definition \ref{wtG}} or \emph{\ref{wtG2}} in
case assumption \emph{(a)} or \emph{(b)} is satisfied respectively,   the following holds:

\begin{itemize}
\item[(i)] there exists  a surface  $\Sigma_{0\a}^{\e}\subset\Sigma_\a^{\e}$ such that
\[\chi(\Sigma_{0\a}^{\e})=\chi(\Sigma\cup\alpha),\]
for some absolutely area minimizing surface $\Sigma$ with $\partial\Sigma=\Gamma$ and where $\chi$ denotes the Euler characteristic, and furthermore
\begin{itemize}
\item under assumption \emph{(a)}, $\Sigma_{0\a}^{\e}$ has two boundary components and
\item under assumption \emph{(b)}, $\Sigma_{0\a}^{\e}$ has a unique boundary component.
\end{itemize}
\item[(ii)] In particular, if $\Gamma$ is smooth then the surfaces $\Sigma_{0\a}^{\e}$ in (ii) can be taken to be equal to $\Sigma_{\a}^{\e}$ and thus we have that
\[\chi (\Sigma_\a^{\e})=\chi(\Sigma\cup\a).\]
\end{itemize}
\end{lem}

\begin{pf}
We prove this lemma by contradiction. Assume the lemma is not true, then for any $\e_\a$, there exists $0<\e<\e_a$ such that the conclusion (i) of
the lemma (and in particular conclusion (ii), in case $\Gamma$ is smooth) is not true. Hence there exists a sequence $\{\e_i\}_{i\in\N}$, with
 $\e_i\downarrow 0$ such that for all $i$, the lemma with $\e=\e_i$ fails.

Let $\Sigma_\a^{\e_i}$ be a sequence of absolutely area minimizing surfaces with $\partial\Sigma_\a^{\e_i}=\widehat\Gamma_\a^{\e_i}$. Then by Lemma
\ref{bridge principle} and in particular Remark \ref{C2convrmk}, there exists an absolutely area minimizing surface  $\Sigma$, with
$\partial\Sigma=\Gamma$, such that after passing to a subsequence, $\Sigma_\a^{\e_i}\to\Sigma$, with the convergence being with respect to the
Hausdorff distance, the measure and also it is a smooth convergence in compact sets of $\ov M\setminus(\Gamma\cup \a)$ and it is also smooth around
points of $\Gamma\setminus\alpha$, where $\Gamma$ is smooth. Finally we let $g_0=\genus(\Sigma)<\infty$

Let now $\Gamma_0\subset\Sigma$ be a smooth curve on $\Sigma$, such that $\Gamma_0=\Gamma$ in $W\supset\a$, $\Gamma_0\sim\Gamma$ and such that the
part of $\Sigma$ bounded by $\Gamma_0$, which we call $\Sigma_0$, has genus $g_0$. Note that $\Sigma_0$ is absolutely area minimizing and it is smooth
(by Theorem \ref{absmincur} and \cite{HSi}). We furthermore take $\Gamma_0$, so that it is not identical to $\Gamma$, so that $\Sigma_0$ is strictly contained in $\Sigma$ and therefore it is strictly stable (because $\Sigma$ is absolutely area minimizing), (cf. \cite{FiSc}).

Now, because of the fact that the convergence $\Sigma^{\epsilon_i}_\alpha\to \Sigma$ is smooth in compact
sets of $\overline M\setminus (\Gamma\cup\a)$ and also around points of $\Gamma\setminus \a$ where $\Gamma$ is smooth, we can find curves
$\widehat\Gamma_{0\a}^{\e_i}\subset \Sigma_\a^{\e_i}$  such that

 \begin{itemize}
 \item Under hypothesis (a), i.e. when $\Gamma$ is a single simple closed curve and $\a$ a $\Gamma$-admissible path, $\widehat\Gamma_{0\a}^{\e_i}\subset \Sigma_\a^{\e_i}$
 consists of two connected components and under hypothesis (b), i.e. when $\Gamma=\Gamma_1\cup\Gamma_2$  and $\a$ is a path connecting them,
 $\widehat\Gamma_{0\a}^{\e_i}\subset \Sigma_\a^{\e_i}$ consists of a single connected component.
 \item $\widehat\Gamma_{0\a}^{\e_i}=\widehat\Gamma_{\a}^{\e_i}$ in $W\supset\widehat\Gamma_{\a}^{\e_i}\cap S_\a^{\e_i}$
 \item $\widehat\Gamma_{0\a}^{\e_i}$ is smooth away from $S_\a^{\e_i}$ (hence $\widehat\Gamma_{0\a}^{\e_i}$ is smooth) and
 \item $\widehat\Gamma_{0\a}^{\e_i}\to \Gamma_0$ smoothly away from $\a$.
 \end{itemize}

Let $\Sigma^{\e_i}_{0\a}$ be the part of $\Sigma^{\e_i}_{\a}$ bounded by $\widehat\Gamma_{0\a}^{\e_i}$. Then, because of the convergences $\Sigma_\a^{\e_i}\to\Sigma$ and $\widehat\Gamma_{0\a}^{\e_i}\to \Gamma_0$, as described above, we have that $\Sigma^{\epsilon_i}_{0\alpha}$ converges to $\Sigma_0$ and the convergence $\Sigma^{\epsilon_i}_{0\alpha}\to \Sigma_0$ is smooth in compact
sets of $\overline M\setminus \a$.

Furthermore, letting
$\Gamma_i=\partial(\Sigma^{\e_i}_{0\a}\cup S_\a^{\e_i})$, we have that  $\Gamma_i$ is smooth, $\widehat\Gamma_{0\a}^{\e_i}$ is the curve that we obtain by
attaching a bridge around $\a$ on the curve $\Gamma_i$ as described in Definitions \ref{wtG} and \ref{wtG2} and $\Gamma_i\to\Gamma_0$ smoothly.

Note, that since $\Sigma_0$ is strictly stable, by White's bridge principle, and in particular by (i) of Theorem \ref{white4}
 (see also Remark \ref{rmkonwhite4}), there exists an open set $U$ containing $\Sigma_0\cup\a$, such that $\Sigma_0$ is the unique area
minimizing surface with boundary $\Gamma_0$ in $\overline U$. Since the convergence $\Sigma^{\epsilon_i}_{0\alpha}\to\Sigma_0$ is also with respect to the
Hausdorff distance, for any $\eta>0$, there exists  $i_0$ such that for any $i\ge i_0$, $\Sigma^{\epsilon_i}_{0\alpha}$ is in $\eta$-neighborhood of
$\Sigma_0\cup\alpha$. Taking $\eta$ small enough (so that the $\eta$-neighborhood of $\Sigma_0\cup\alpha$ is contained in $U$) we have that there exists
$i_0$ such that for any $i\ge i_0$, $\Sigma^{\epsilon_i}_{0\alpha}\subset U$. Since $\Sigma^{\epsilon_i}_{0\alpha}$ are absolutely area minimizing they
also minimize area among all surfaces with the same boundary in $\overline U$.

Let $\Sigma_i$ now be an area minimizing surface in $U$, such that $\partial \Sigma_i=\Gamma_i$. Then $\area(\Sigma_i)\le
\area(\Sigma^{\e_i}_{0\a})+\area (S_\a^{\e_i})\le  \area(\Sigma^{\e_i}_{\a})+\area (S_\a^{\e_i})$, so after passing to a subsequence, $\Sigma_i$
converge to an area minimizing surface in $U$ with boundary equal to $\Gamma_0$, and since $
\Sigma_0$ is the unique such surface, we have that $\Sigma_i\to\Sigma_0$. Arguing as in the proof of Lemma
\ref{bridge principle} and in particular as in Remark \ref{C2convrmk}, we can apply White's curvature estimate, Theorem \ref{white2c}, to conclude that the convergence is smooth.

Note now that $\Sigma^{\e_i}_{0\a}$ is absolutely area minimizing and thus by White's bridge principle, in particular by (ii) of Theorem \ref{white4}, for $i$ large enough
$\Sigma_{0\a}^{\e_i}$ and $\Sigma_i\cup S_\a^{\e_i}$ are diffeomorphic. Since $\Sigma_i\to\Sigma_0$ smoothly and  $(\Sigma_0\cup
S_\a^{\e_i})\sim(\Sigma_0\cup\a)$ we have that $\Sigma_{0\a}^{\e_i}$ is homotopic to $\Sigma_0\cup\a$, which by construction is homotopic to
$\Sigma\cup\alpha$, and thus $\chi(\Sigma_{0\a}^{\e_i})=\chi(\Sigma\cup\a)$.  This leads to a contradiction, since we have assumed that (i) of the
lemma with $\e=\e_i$ fails for any $i$ and thus finishes the proof of part (i) of the lemma.

If $\Gamma$ is smooth, then we can pick $\Gamma_0$, so that $\Sigma\setminus\Sigma_0$ is topologically a disk. Furthermore, since $\Gamma_0=\Gamma$ in $W\supset \a$, there exists an open set $\mathcal O\subset M\setminus \a$, such that $\mathcal O\supset \Sigma\setminus\Sigma_0$,   $\mathcal O\supset \Sigma_{\a}^{\e_i}\setminus\Sigma_{0\a}^{\e_i}$ for $i$ large enough and the convergence $\Sigma_\a^{\e_i}\to\Sigma$
is smooth in $\mathcal O$. This, along with the convergence $\Sigma_{0\a}^{\e_i}\to\Sigma_0$ and the fact that $\Sigma\setminus\Sigma_0$ is topologically a disk, implies that for large $i$, $\Sigma_{\a}^{\e_i}\setminus\Sigma_{0\a}^{\e_i}$ is also topologically a disk and therefore $\chi(\Sigma_{\a}^{\e_i})=\chi(\Sigma_{0\a}^{\e_i})=\chi(\Sigma\cup\a)$.  This leads also to a
contradiction, since we have assumed that if $\Gamma$ is smooth, then (ii) of the lemma with $\e=\e_i$ fails for any $i$ and thus finishes the proof
of the lemma.

\end{pf}

\begin{cor}\label{genuscor}
Let $\Gamma, \a$ satisfy assumption \emph{(a)} or \emph{(b)} of \emph{Theorem \ref{bridge principle}}. We further assume that $\a$ is smooth,
$\Gamma$ is smooth in an open set $W$ containing $\a$ and for any absolutely area minimizing surface $\Sigma$, with $\partial\Sigma=\Gamma$,
 $\Sigma$ is connected  and $\genus(\Sigma)<\infty$. Then, for almost every $0<\e<\e_a$ , where $\e_a$ is as in \emph{Lemma \ref{genuscor}}, and for any absolutely area
minimizing surface $\Sigma_\a^{\e}$, with $\partial\Sigma_\a^{\e}=\widehat\Gamma_\a^\e$, where $\widehat\Gamma_\a^\e$ is as in \emph{Definition \ref{wtG}}
or \emph{\ref{wtG2}} in case assumption \emph{(a)} or \emph{(b)} is satisfied respectively, the following holds:

\begin{itemize}
\item[(i)]
\begin{itemize}
\item under assumption \emph{(a)}
\[\genus (\Sigma_\a^{\e})\ge\genus(\Sigma),\]
\item under assumption \emph{(b)}
\[\genus (\Sigma_\a^{\e})\ge\genus(\Sigma)+1\]
\end{itemize}
for some absolutely area minimizing surface $\Sigma$ with $\partial\Sigma=\Gamma$ .
\item[(ii)] If in particular $\Gamma$ is smooth, then
\begin{itemize}
\item under assumption \emph{(a)}
\[\genus (\Sigma_\a^{\e})=\genus(\Sigma),\]
\item under assumption \emph{(b)}
\[\genus (\Sigma_\a^{\e})=\genus(\Sigma)+1.\]
\end{itemize}
\end{itemize}
\end{cor}

\begin{proof}

Recall that by the classification of orientable surfaces, $S$ is a genus $g$ surface with $k$
boundary components if and only if $\chi(S)= (2-k)-2g$.

Let $\Sigma^{\e}_{0\alpha}\subset \Sigma^{\e}_0$ be as in (ii) of Lemma \ref{genuslem}, so that
\[\chi(\Sigma^{\e}_{0\alpha})=\chi(\Sigma\cup\alpha)\]
for some absolutely area minimizing surface $\Sigma$ with $\partial\Sigma=\Gamma$ and so that $\Sigma_{0\alpha}^\e$ has two boundary components in case assumption (a) is satisfied and a unique boundary component in case assumption (b) is satisfied. Let also $g_0=\genus(\Sigma)<\infty$.

By realizing the arc $\alpha$ as an edge in the triangulation, we have
$\chi(\Sigma\cup{\alpha}) = \chi(\Sigma) -1$.

Now under hypothesis (a),  $\Sigma$ is a genus $g_0$ surface with one boundary component, therefore $\chi(\Sigma\cup{\alpha}) = \chi(\Sigma)
-1=-2g_0$. Since, $\Sigma^{\epsilon}_{0\alpha}$ has two boundary components and $\chi(\Sigma_{0\a}^{\e})=\chi(\Sigma\cup\a)=-2g_0$, we have that
$\Sigma^{\epsilon}_{0\alpha}$ is a genus $g_0$ surface with two boundary components.

On the other hand, under hypothesis (b), as $\Sigma$ is a connected surface with two boundary components, $\chi(\Sigma\cup{\alpha}) = \chi(\Sigma)
-1=-2g_0-1$ where $g_0$ is the genus of $\Sigma$. Since, $\Sigma^{\epsilon}_{0\alpha}$ has a unique boundary component and
$\chi(\Sigma_{0\a}^{\e})=\chi(\Sigma\cup\a)=-2g_0-1$, we have that $\Sigma^{\epsilon}_{0\alpha}$ is a genus $g_0+1$ surface with one boundary
component.

Note now that since $\Sigma_{0\a}^{\e}\subset\Sigma_\a^{\e}$, $\genus(\Sigma_\a^{\e})\ge\genus(\Sigma_{0\a}^{\e})$. This finishes the proof of (i) of
the lemma.

 Part (ii) is immediate from (ii) of Lemma \ref{genuslem}, since if $\Gamma$ is smooth, then $\Sigma_{0\a}^{\e}=\Sigma_\a^{\e}$.
\end{proof}

\subsection{Horn Surgery and Thin Handles}\label{hssection} \

\

In the previous part, by adding a bridge near a $\Gamma$-admissible path $\alpha$ to the simple closed curve $\Gamma$, we get
$\widehat{\Gamma}^\epsilon_\alpha=\Gamma\sharp S_\a^\e$ (as in Definition \ref{wtG}) which is the union of two simple closed curves
$\widehat{\Gamma}^{\epsilon,1}_\alpha$ and $\widehat{\Gamma}^{\epsilon,2}_\alpha$ in $\partial M$. Now,  we repeat this process one more time, by
connecting the two simple closed curves $\widehat{\Gamma}^{\epsilon,1}_\alpha$ and $\widehat{\Gamma}^{\epsilon,2}_\alpha$ via a bridge around a new
path $\tau$, connecting the two curves. Thus, we get a simple closed curve
$\widehat{\Gamma}^{\epsilon,\delta}_{\alpha,\tau}=\widehat\Gamma_{\a}^\e\sharp S_\tau^\delta$ (as in Definition \ref{a2} and where $\d$ here denotes the width of the new bridge) which
bounds an absolutely area minimizing surface whose genus, we will show, is strictly greater than the genus of the absolutely area minimizing surface
$\Sigma$ in $M$ with $\partial \Sigma = \Gamma$. In other words, by adding two bridges $S_\epsilon(\alpha)$ and $S_\delta(\tau)$ successively to a
simple closed curve $\Gamma$, we modify the absolutely area minimizing surface it bounds by adding a thin handle. We will refer to this operation on
$\Gamma$ as {\em the horn surgery} (See Figure 4), and to the resulting additional handle in the new absolutely area minimizing surface as {\em the
thin handle} (See Figure 5). Now, let's give the formal construction and definitions.

\begin{figure}[t]

\relabelbox  {\epsfxsize=5in

\centerline{\epsfbox{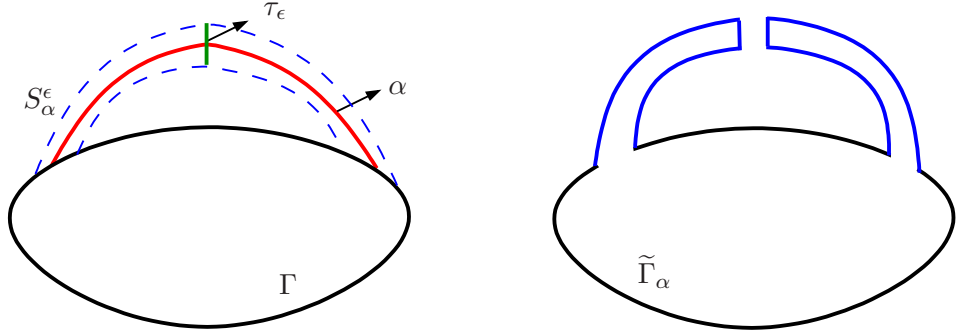}}}

\relabel{1}{$\Gamma$}

\relabel{2}{$S^\epsilon_\alpha$}

\relabel{3}{$\tau_\epsilon$}

\relabel{4}{$\alpha$}

\relabel{5}{$\widetilde{\Gamma}_\alpha$}

\endrelabelbox

\caption{\label{fig:figure2} \small {The Horn Surgery on $\Gamma$ along $\alpha$. }}

\end{figure}

We will use the notation introduced in  Section \ref{bridgeprinciplesubsection}. Let $\Gamma$, $\alpha$ and $\widehat{\Gamma}^\epsilon_\alpha$ be
as in Definition \ref{wtG}. Let $z$ be the midpoint of $\alpha$. Let, $z^+_\epsilon, z^-_\epsilon$ be the midpoints of
$\alpha^+_\epsilon,\alpha^-_\epsilon$ respectively. Let $\tau_\epsilon$ be a small path from $z^+_\epsilon$ to $z^-_\epsilon$ through $z$ in
$\partial M$. For $0<\delta\ll \epsilon$, let $S^\delta_\tau= N_\delta(\tau_\epsilon)\cap S^\epsilon_\alpha$ be a tiny strip (``bridge") of width
$\d$ around $\tau_\epsilon$ connecting the simple closed curves $\widehat{\Gamma}^{\epsilon,1}_\alpha$ and $\widehat{\Gamma}^{\epsilon,2}_\alpha$
where $N_\delta(\tau_\epsilon)$ is the $\delta$ neighborhood of $\tau_\epsilon$ in $\partial M$. Here, we assume $\delta$ is sufficiently small so
that for any $0<\delta'<\delta$, $\partial S^{\delta'}_\tau \cap \inr(S^\epsilon_\alpha)$ consists of exactly two arc segments, which we call
$\tau^+_\epsilon$ and $\tau^-_\epsilon$. Then $\partial S^\delta_\tau = (\tau^+_\epsilon \cup \tau^-_\epsilon) \cup (\gamma^1_\delta \cup
\gamma^2_\delta)$ where $\gamma^i_\delta \subset \widehat{\Gamma}^{\epsilon,i}_\alpha$, $i=1,2$.
\begin{defn}\label{hs}
We define {\em the horn surgery} on $\Gamma\in\A$ along a $\Gamma$-admissible path
$\alpha$ to be the curve
\[\widetilde{\Gamma}_\alpha = \Gamma\sharp S^\epsilon_\alpha\sharp S^\delta_\tau\]
with $\tau$ being a path passing through the midpoint of $\a$ as described above (recall Definitions \ref{wtG}, \ref{wtG2}). With the above notation this curve is given \[\widetilde{\Gamma}_\alpha=(\widehat{\Gamma}^\epsilon_\alpha \setminus
(\gamma^1_\delta \cup \gamma^2_\delta)) \cup (\tau^+_\epsilon \cup \tau^-_\epsilon).\] Hence, $\widetilde{\Gamma}_\alpha$ is a simple closed curve in
$\partial M$ (See Figure 4).
\end{defn}

\begin{rmk}\label{smoothglueII}
By Remark \ref{smoothglue}, we note here that if $\a$ is smooth and the curve $\Gamma$ is smooth around $\a$, then $\widehat\Gamma_\a^\e$ is smooth
around  $\widehat\Gamma_\a^\e\cap S^\e_\a$ and in particular around $\widehat\Gamma_\a^\e\cap\tau$. Then we can take the path $\tau$ to be smooth and we can consider gluing the bridge $S^\d_\tau$ smoothly
along $\widehat\Gamma_\a^\e$, so that the resulting curve $\widetilde{\Gamma}_\a$ is smooth around $\widehat\Gamma_\a^\e\cap S^\d_\tau$ and hence it is also smooth around $\widehat\Gamma_\a^\e\cap S^\e_\a$.

In the rest of the paper and as long as these smoothness assumptions are satisfied we will always consider constructing the curve $\wt{\Gamma}_\a$ to
be  smooth around $\a$.
\end{rmk}

Let $\Gamma\in \A_g$, and $\alpha$ be a $\Gamma$-admissible path. Now, we claim that after the horn surgery on $\Gamma$ along $\alpha$,
$\widetilde{\Gamma}_\alpha$ is in $\A_{g+1}$. In other words, the horn surgery on the boundary curve $\Gamma$ will add a thin handle to an area
minimizing surface $\Sigma$ with boundary $\Gamma$ and thus increase its genus (See Figure 5).

\begin{figure}[t]

\relabelbox  {\epsfxsize=5in

\centerline{\epsfbox{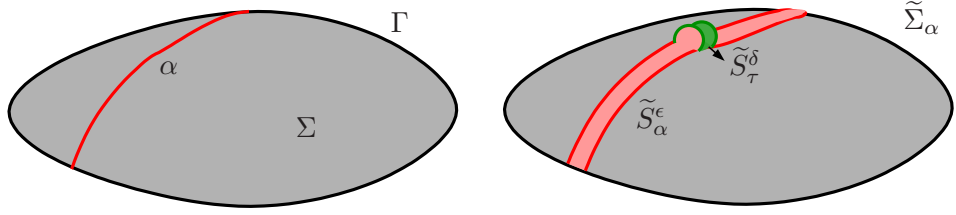}}}

\relabel{1}{$\Gamma$}

\relabel{2}{$\Sigma$}

\relabel{3}{$\alpha$}

\relabel{5}{$\widetilde{\Sigma}_\alpha$}

\relabel{6}{$\widetilde{S}^\delta_\tau$}

\relabel{7}{$\widetilde{S}^\epsilon_\alpha$}

\endrelabelbox

\caption{\label{fig:figure2} \small {Adding a {\em Thin Handle} to $\Sigma$ along $\alpha$. Here, $\widetilde{S}^\epsilon_\alpha$ and
$\widetilde{S}^\delta_\tau$ represent the parts of $\widetilde{\Sigma}_\alpha$, which is close to the strips $S^\epsilon_\alpha$ and $S^\delta_\tau$
in $\partial M$.}}

\end{figure}

\begin{lem}\label{thinhandle} {\bf (Thin handle)}
Let $\Gamma\in\A_g$, and $\alpha$ be a $\Gamma$-admissible smooth path, such that $\a$ is smooth, $\Gamma$ is smooth in an open set $W$ containing
$\a$ and for  any absolutely area minimizing surface $\Sigma$ with $\partial\Sigma=\Gamma$, $\genus(\Sigma)<\infty$.
 Then there exists $\e_a$, such that for almost every
$0<\epsilon<\epsilon_\alpha$ the following holds:
\begin{itemize}
\item[(i)] There is a sufficiently small $\delta_{\tau}>0$, such that and for any $0<\delta<\delta_{\tau}$, after the horn surgery on
$\Gamma$ along $\alpha$ (see \emph{Definition \ref{hs}}), $\widetilde{\Gamma}_\alpha\in\A_{g+1}$.
In particular, for any absolutely area minimizing surface $\wt\Sigma_\a$ with $\partial\wt\Sigma_\a=\wt\Gamma_\a$,   $\genus(\wt\Sigma_\a)\ge\genus(\Sigma)+1$, where
$\Sigma$ is an absolutely area minimizing surface with $\partial\Sigma=\Gamma$.
\item[(ii)] If $\Gamma$ is smooth, then  $\widetilde{\Gamma}_\alpha$ is a smooth curve in $\A_{g+1}$ and in particular, for any absolutely area minimizing surface
$\wt\Sigma_\a$ with $\partial\wt\Sigma_\a=\wt\Gamma_\a$,   $\genus(\wt\Sigma_\a)=\genus(\Sigma)+1$, where $\Sigma$ is an absolutely area minimizing surface with
$\partial\Sigma=\Gamma$.

\end{itemize}
\end{lem}

\begin{pf}
Recall (Remarks \ref{smoothglue}, \ref{smoothglueII}) that $\widehat \Gamma_\a^\e$ and $\widetilde{\Gamma}_\alpha$ are  both smooth around $S^\e_\a$. Note first that for any $\e<\e_a$, with $\e_\a$ as in  Corollary \ref{genuscor},  and for any absolutely area minimizing
surface $\Sigma_\a^\e$ with $\partial\Sigma_\a^\e=\widehat\Gamma_\a^\e$
\[\genus(\Sigma_\a^\e)\ge \genus(\Sigma)\]
for an absolutely area minimizing surface $\Sigma$ with $\partial\Sigma=\Gamma$ and the above inequality is actually an equality if $\Gamma$ is smooth.

We apply now again Corollary \ref{genuscor} with $\Gamma$, $\a$, $\Sigma_\a^\e$ and $\Sigma$ replaced by $\widehat\Gamma_\a^\e$, $\tau$,
$\wt\Sigma_\a$ and $\Sigma_\a^\e$ (note that here $\Sigma_\a^\e$, $\wt\Sigma_\a$ denote {\it{any}} absolutely area minimizing surfaces with boundary
equal to $\widehat\Gamma_\a^\e$ and $\wt\Gamma_\a$ respectively). Since now we use the hypothesis (b) of Corollary \ref{genuscor}, we get that there
exists a $\delta_\tau>0$ such that for any $0<\d<\d_\tau$ and for any absolutely area minimizing surface $\wt\Sigma_\a$, with
$\partial\wt\Sigma_\a=\wt\Gamma_\a$
\[\genus(\wt\Sigma_\a)\ge\genus(\Sigma_\a^\e)+1\ge \genus(\Sigma)+1\]
for an absolutely area minimizing surface $\Sigma_\a^\e$ with boundary equal to $\widehat\Gamma_\a^\e$ and an absolutely area minimizing surface
$\Sigma$ with boundary equal to $\Gamma$. And in the case when $\Gamma$ is smooth (cf. Remark \ref{smoothglueII}) the above inequalities are actually
equalities and thus
\[\genus(\wt\Sigma_\a)=\genus(\Sigma_\a^\e)+1= \genus(\Sigma)+1.\]

\end{pf}

\begin{rmk} In the horn surgery construction, the $\Gamma$-admissible path $\alpha$ determines the $\epsilon_\alpha>0$, and the
path $\tau$ determines the $\delta_\tau>0$. Since $\tau$ depends on the $0<\epsilon<\epsilon_\alpha$, $\delta_\tau$ also depends on
$\epsilon_\alpha$. From now on, we will assume the choices for the horn surgery on $\Gamma$ along $\alpha$ are canonical. In other words, we always
consider $0<\epsilon<\epsilon_\a$, $\tau$ and $0<\delta<\delta_\tau$, so that the horn surgery is well defined. Furthermore from now on we assume
that $\e_\a$ and $\d_\tau$ are small enough so that Lemma \ref{thinhandle} holds, i.e.  if $\Gamma$ is smooth around $\a$, then   for
$\widetilde{\Gamma}_\alpha = \Gamma\sharp S^\epsilon_\alpha\sharp S^\delta_\tau$ we can apply Lemma \ref{thinhandle} and in particular if $\Gamma$ is
smooth, then we can apply (ii) of Lemma \ref{thinhandle}.

\end{rmk}

Now, let's consider the space $\B_g$. We show that $\B_g$ is nonempty for any $g\geq 0$. As noted before, $\B_g = \A_g \setminus \A_{g+1}$. In Lemma \ref{Agopen} we proved that
 for any $g$, $\A_g$ is open, but this does not imply that $\B_g$ is nonempty. In order to show this, we need to rule out the case $\A_g =
\A_{g+1}$. In other words, we will show that for any $g\geq 0$, there is a simple closed curve $\Gamma_g\in \A$ such that the minimum genus among all
 absolutely area minimizing surface with boundary $\Gamma_g$ is exactly $g$.

\begin{thm}\label{Bgne}
For any $g\geq 0$, $\B_g$ is nonempty.
\end{thm}

\begin{pf} Since $\B_g=\A_g\setminus\A_{g+1}$, it suffices to show the existence of a simple closed curve $\Gamma_g\in \A_g$ such that $\Gamma_g$ bounds an absolutely area
minimizing surface $\Sigma$ of genus $g$.

Clearly, $\B_0$ is not empty, as we can take a sufficiently small smooth simple closed curve in $\partial M$, so that the absolutely area minimizing surface
bounded by this curve is a smooth disk in $M$, say $D_0$, and since $M$ is strictly mean convex, $D_0$ is also properly embedded, i.e. $D_0\cap\partial M=\partial D$.

Now, let $\Gamma_0$ be a smooth curve in $\B_0$ and take a $\Gamma_0$-admissible and smooth path $\alpha_1$. Then, by (ii) of Lemma \ref{thinhandle},
the horn surgery on $\Gamma_0$ along $\alpha_1$ will give us a simple closed and smooth curve $\widetilde{\Gamma_0}_{\alpha_1}$, say $\Gamma_1$, such
that $\Gamma_1 \in \B_1$.

We can argue now by induction, that for any $k$ there exists a smooth curve in $\B_k$ as follows: Assume that there exists a smooth curve
$\Gamma_{k-1}\in\B_{k-1}$ and take a $\Gamma_{k-1}$-admissible and smooth path $\alpha_k$. The horn surgery on $\Gamma_{k-1}$ along $\alpha_k$ (Lemma
\ref{thinhandle}) gives a new smooth curve $\Gamma_k$ with $\Gamma_k \in \B_k$ and thus $\B_k\ne\emptyset$.
\end{pf}

\begin{rmk}\label{Bgnermk} Intuitively, this gives a construction of adding a handle to an absolutely area minimizing surface by
modifying the boundary curve. Note that, since the handles can be as small as we want, the area of the new surfaces can be as close as we want to the
area of the original surface (cf. Lemma \ref{bridge principle}). Furthermore, by the proof of Theorem \ref{Bgne}, we actually conclude that not only
can we always find a simple closed curve in $\B_g$, but also we can find a smooth simple closed curve in $\B_g$.
\end{rmk}

In Lemma \ref{Agopen}, we proved that for any $g>0$, $\A_g$ is open in $\A$, and in Theorem~\ref{Bgne}, we showed that $\B_g=\A_g\setminus\A_{g+1}$
is nonempty.  Now, by using the handle construction in Lemma \ref{thinhandle}, we will prove that $\A_g$ is not only open in $\A$, but also dense in
$\A$.

\begin{thm}\label{Adense}
For any $g\geq 0$, $\A_g$ is dense in $\A$, i.e. $\overline{\A}_g = \A$.
\end{thm}

\begin{pf}
Fix $g>0$. It suffices to show that for any $\Gamma_0 \in \A$, and for any $\mu>0$, there exists a simple closed curve $\Gamma\in \A_g$ with
$d(\Gamma,\Gamma_0)<\mu$, where by $d(\Gamma,\Gamma_0)<\mu$ we mean that $\Gamma$ is in a $\mu$-neighborhood of $\Gamma_0$ in the $\mathcal C^0$
topology  (cf. Definition~\ref{c0rmk}).

Let $\Gamma_0\in \A$. If $\Gamma_0\in\A_g$, then we are done as we can take $\Gamma=\Gamma_0$. Hence we can assume that $\Gamma_0\notin \A_g$, which
implies that for some $g_0$, with $0\leq g_0\leq g-1$, $\Gamma_0$ belongs to $\B_{g_0}\subset\A_{g_0}$. Since $\A_{g_0}$ is open we can take
$\Gamma_1\in \A_{g_0}$, so that $\Gamma_1$ is piecewise smooth and $d(\Gamma_0, \Gamma_1)<\mu/2$.

If $\Gamma_1\in\A_g$, then we can take $\Gamma=\Gamma_1$ and we are done. Therefore, we assume that $\Gamma_1\in\B_{g_1}$, with $g>g_1\ge g_0$ and we
will  construct $\Gamma$ by $h=g-g_1$ horn surgeries on $\Gamma_1$.

Recall the construction in Lemma \ref{thinhandle}. Since $\Gamma_1$ is piecewise smooth, we can find a $\Gamma_1$-admissible and smooth path
$\alpha_1$  satisfying the hypothesis of Lemma \ref{thinhandle}, i.e. so that $\Gamma_1$ is smooth in an open neighborhood of $\a_1$. This implies
that  since $\Gamma_1\in \A_{g_1}$,  $\widetilde{\Gamma_1}_\alpha \in \A_{g_1+1}$. Let $\Gamma_2= \widetilde{\Gamma_1}_\alpha$, so that $\Gamma_2$ is
obtained by horn surgery on $\Gamma_1$ along the $\Gamma_1$-admissible path $\alpha_1$. In general we let $\Gamma_k$ to be the curve obtained by horn
surgery on $\Gamma_{k-1}$ along a $\Gamma_{k-1}$-admissible smooth path ${\alpha_k}$, satisfying the hypothesis of Lemma \ref{thinhandle}. Then
$\Gamma_k\in\A_{g_1+k-1}$, $\Gamma_{k}\sim\Gamma_{k-1}$ and since we are free to choose the admissible paths $\a_k$ as short as we want, we can take
the $\Gamma_k$'s so that, $d(\Gamma_{k-1}, \Gamma_k) < \displaystyle \frac{1}{2h}\mu$, for any given $\mu>0$. Let $\Gamma=\Gamma_{h+1}$. Then, by
construction $\Gamma_h \in  \A_{g_1+h}=\A_{g}$, and $d(\Gamma_0,\Gamma) < \frac{\mu}{2}+h\frac{\mu}{2h}=\mu$. The proof follows.
\end{pf}

\begin{figure}[b]

\relabelbox  {\epsfxsize=2in

\centerline{\epsfbox{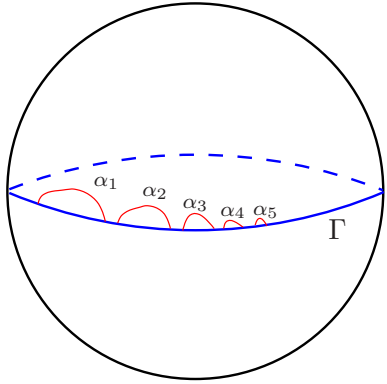}}}

\relabel{1}{$\Gamma$}

\relabel{2}{\scriptsize $\alpha_1$}

\relabel{3}{\scriptsize $\alpha_2$}

\relabel{4}{\scriptsize $\alpha_3$}

\relabel{5}{\scriptsize $\alpha_4$}

\relabel{6}{\scriptsize $\alpha_5$}

\endrelabelbox

\caption{\label{fig:figure2} \small {Construction of a rectifiable curve in $\partial M$ bounding an absolutely area minimizing surface of infinite
genus in $M$ by doing infinitely many horn surgeries.}}

\end{figure}

\begin{cor} \label{Bndense} For any $g\geq 0$, $\B_g$ is nowhere dense in $\A$.
\end{cor}

\begin{pf} $\B_g=\A_g\setminus\A_{g+1}$ and thus $\overline\B_g = \overline\A_g\setminus\inr(A_{g+1}) =\A \setminus A_{g+1}$, since $\A_g$ is dense in $\A$ by
Theorem \ref{Adense} and $\A_{g+1}$ is open in $\A$, by Lemma \ref{Agopen}. Since $\A_{g+1}$ is also dense in $\A$,
$\inr(\overline{\B_g})=\emptyset$.
\end{pf}

\begin{rmk}\label{genusinfty} Note  that $\B_\infty:=\bigcap_g \A_g \neq \emptyset$. To see this, even though $\B_\infty=\bigcap_g \A_g$, where all the $\A_g$
are open and dense in $\A$, the classical Baire Category Theorem argument ($\bigcap_g \A_g$ is dense in $\A$)  would not work since $\A$ is not a
complete metric space. However, it is possible to construct a simple rectifiable curve $\Gamma_\infty$ in $\partial M$ which bounds an absolutely
area minimizing surface $\Sigma$ with infinite genus. To see this, take a simple closed curve $\Gamma$ in $\partial M$, and do the horn surgery
infinitely many times by choosing sufficiently small admissible paths $\{\alpha_1, \alpha_2, ..\}$ successively (See Figure 6).

Note also that since the above construction can be done on any piecewise smooth $\Gamma\in\A$, we actually have that $\B_\infty$ is dense in $\A$.
\end{rmk}

To summarize, so far in this section, we have shown that $\A=\A_0\varsupsetneq\A_1\varsupsetneq\A_2\varsupsetneq ...\varsupsetneq A_g \varsupsetneq ...$
where $\A_g$ is an open dense subset of $\A$ for any $g\geq 0$. Moreover, we have shown  that $\B_g=\A_g\setminus\A_{g+1}$ is not empty, and nowhere dense in $\A$
for any $g\geq 0$.

We finish this section with a result which states that there are simple closed curves bounding absolutely area minimizing surfaces with different
genus. In other words, we show that the relation $\mathfrak{g}:\A\to \BN$ such
that $\mathfrak{g}(\Gamma)$ is defined to be the genus of an absolutely area minimizing surface $\Sigma_\Gamma$ in $M$ with $\partial \Sigma_\Gamma=\Gamma$,
is not a function.

\begin{thm} \label{diffgenus} Let $M$ be a compact, orientable, strictly mean convex $3$-manifold. Then, for any $g\geq 0$, there exists a simple closed curve
$\Gamma$ in $\partial M$ such that $\Gamma$ bounds two different absolutely area minimizing surfaces $\Sigma_1$ and $\Sigma_2$ in $M$ with
$g=\text{genus}(\Sigma_1)\ne \text{genus}(\Sigma_2)$.
\end{thm}

\begin{pf} Let $\Gamma$ be a smooth curve in $\B_g$ (cf. Theorem \ref{Bgne}, Remark \ref{Bgnermk}). Let $\alpha$ be a smooth and  $\Gamma$-admissible path
in $\partial M$. Then, by Lemma \ref{bridge principle}, for any $\eta_0>0$, there is an $\epsilon_0>0$ such that any absolutely area minimizing
surface $\Sigma_\alpha$ with boundary $\widehat{\Gamma}^{\epsilon_0}_\alpha$, is $\eta_0$-close to the surface $\Sigma\cup S^{\epsilon_0}_\alpha$, for
some absolutely area minimizing surface $\Sigma$, with $\partial\Sigma=\Gamma$ (where $\widehat\Gamma_\a^{\e_0}, S_\a^{\e_0}$ are as in
Definition~\ref{wtG}). Recall that since $\Gamma$ is smooth, we consider this bridge construction so that $\widehat\Gamma_\a^{\e_0}$ is smooth (cf.
Remark \ref{smoothglue}). Now foliate the strip $S^{\epsilon_0}_\alpha$ with arcs $\{\tau_{\delta}\}_{|\delta|\leq c}$ where $\tau_\delta$ is
parametrized by the arclength of $\alpha$, with $|\alpha|=2c$, such that $\tau_0 = \tau$ and $\tau_{\pm c} \subset \Gamma$. We also construct the
arcs $\tau_\d$, so that they are smooth and hence after the horn surgery, the curves $\widetilde{\Gamma}^\d_\alpha =\Gamma\sharp S_\a^{\e_0}\sharp
S_\tau^\delta=(\widehat{\Gamma}^{\epsilon_0}_\alpha \setminus (\gamma^1_\delta \cup \gamma^2_\delta)) \cup (\tau_{\d} \cup \tau_{-\d})$ are smooth,
where  $S^\delta_\tau$ represents the strip in $S^{\epsilon_0}_\alpha$ separated by the arcs $\tau_\delta$ and $\tau_{-\delta}$ and  $\partial
S_\tau^\delta=\gamma^1_\delta \cup \gamma^2_\delta\cup\tau_{\d}\cup\tau_{-\d}$ (cf. Definition~\ref{hs}, Remark \ref{smoothglueII}). Then
$\widetilde{\Gamma}^c_\alpha = \Gamma$ and $\widetilde{\Gamma}^0_\alpha =\widehat{\Gamma}^{\epsilon_o}_\alpha$.

%



%


%




%

Now we note the following:
\begin{itemize}
\item[(i)] Let $\d\in(0,c]$, $\{\d_i\}_{i\in\N}$ be a sequence in $[0, c]$ such that $\d_i\to\d$ and let $\wt\Sigma_\a^{\d_i}$ be a sequence of absolutely area
minimizing surfaces with $\partial\wt\Sigma_\a^{\d_i}=\wt\Gamma_\a^{\d_i}$. Then, arguing as in Lemma \ref{bridge principle} (see in particular Remark \ref{C2convrmk}), with
the use of the compactness theorem in \cite{FeF} and White's curvature estimate, Theorem \ref{white2c}, we have that after passing to a subsequence, $\wt\Sigma_\a^{\d_i}$ converges smoothly
to $\wt\Sigma_\a^{\d}$, where $\wt\Sigma_\a^{\d}$ is an absolutely area minimizing surface with $\partial\wt\Sigma_\a^{\d}=\wt\Gamma_\a^{\d}$. This, along with the Gauss
Bonnet theorem implies that $\lim_i\genus(\wt\Sigma_\a^{\d_i})=\genus(\wt\Sigma_\a^{\d})$.
\item[(ii)]By Lemma \ref{thinhandle}, there is a $\delta_0>0$ such that for any $0<\delta<\delta_0$, $\widetilde{\Gamma}^\delta_\alpha$ is a smooth curve in
$\B_{g+1}$ and on the other hand we know that $\widetilde{\Gamma}^c_\alpha=\Gamma\in\B_g$. Moreover, by (i) we have that $\d_0<c$.
\end{itemize}

Let $\d_1$ be the infimum of all $\d\in[0,c]$, such that $\widetilde{\Gamma}^\d_\alpha\in\B_g$. By (i) and (ii) above we know that
$\d_1\in[\d_0,c)\subset(0,c)$. Let now $\Gamma_0=\wt{\Gamma}^{\d_1}_\alpha$. Then $\Gamma_0$ is a smooth curve and $\Gamma_0\ne\Gamma$. Applying (i),
with a sequence $\d_i\downarrow \d_1$ and a sequence $\wt\Sigma_\a^{\d_i}$ of absolutely area minimizing surfaces with genus equal to $g$ (which we can do since we can pick the $\d_i$'s, so that for each $i$, $\wt\Gamma^{\d_i}_\a\in\B_g$), we conclude that there exists an absolutely area minimizing surface $\Sigma_1$ with $\partial\Sigma_1=\Gamma_0$ and
$\genus(\Sigma_1)=g$. Applying now (i), with a sequence $\d_i\uparrow \d_1$ and a sequence $\wt\Sigma_\a^{\d_i}$ of absolutely area minimizing surfaces with genus not equal to $g$ (which we can do since for each $i$, $\wt\Gamma^{\d_i}_\a\notin\B_g$), we conclude there
exists an absolutely area minimizing surface $\Sigma_2$ with $\partial\Sigma_2=\Gamma_0$ and $\genus(\Sigma_2)\ne g$. Hence, this curve $\Gamma_0$,
satisfies the hypothesis of the Theorem.
\end{pf}

\section{Results on Minimal Surfaces in Mean Convex Domains}\label{minsurfaces section}

In this section, we will give interesting results on uniqueness, embeddedness and the genus of minimal surfaces which are properly embedded in a mean
convex $3$-manifold. In particular, we generalize the examples of Peter Hall \cite{Ha} in the unit $3$-ball $\BB^3$, addressing the questions by Bill
Meeks, to any strictly mean convex $3$-manifold.

First, by using the results of the previous section, we show that the space $\CC$ of simple closed curves in $\partial M$ bounding more than one
minimal surface in $M$ is generic in the set of simple nullhomologous closed curves in $\partial M$, where $M$ is a strictly mean convex
$3$-manifold.

Before giving the rigorous proof of this theorem, we give an outline of the main argument used. If $\partial M$ is homeomorphic to $S^2$, then the
theorem follows directly from the fact that $\A_1$ is open dense in $\A$, and  $\A_1\subset \CC$. This is because any simple closed curve $\Gamma$
would be nullhomotopic in $M$ as $\partial M \simeq S^2$, and by the result of Meeks and Yau (see Theorem \ref{meeksyau}), $\Gamma$ bounds an area
minimizing disk $D$ in $M$. Hence, any simple closed curve $\Gamma$ in $\A_1$ automatically bounds two different minimal surfaces where one of them
is the absolutely area minimizing surface $\Sigma$ with genus $\geq 1$, and the other one is the area minimizing disk $D$. The same argument works in
general with a slight modification. In particular, we will show that for an open dense subset $\CC$, any simple closed curve $\Gamma$ in this subset
bounds an absolutely area minimizing surface $\Sigma$ in $M$ where $\Sigma$ is not the smallest genus surface which $\Gamma$ bounds in $M$. By a
Theorem of White, which we have stated in Theorem \ref{white2}, there is an area minimizing surface $S$ of smallest genus. This implies $\Gamma$
bounds at least two different minimal surfaces, where one of them is the absolutely area minimizing surface $\Sigma$, and the other one is the area
minimizing representative $S$ of the smallest genus surfaces which $\Gamma$ bounds in $M$.

\begin{thm}\label{nonuniquecurves} {\bf (Curves bounding more than one minimal surface are generic)}

Let $M$ be a strictly mean convex $3$-manifold. Let $\A$ be the space of nullhomologous simple closed curves in $\partial M$ equipped with the
$\mathcal C^0$ topology. Let $\mathfrak{C}\subset \A$ represent the simple closed curves in $\partial M$ bounding more than one embedded stable
minimal surface in $M$. Then, $\mathfrak{C}$ is generic in $\A$ in the sense that $\mathfrak{C}$ contains an open dense subset of $\A$.
\end{thm}

\begin{rmk}\label{Crmk}
In particular we show the following: Let AAM$(\Gamma)$, AM$(\Gamma)$ be the set of Absolutely Area Minimizing and that of Area Minimizing surfaces
respectively bounded by $\Gamma$. Then the set
\[\{\Gamma\in\A:\exists \Sigma_1\in \text{AAM}(\Gamma), \Sigma_2\in \text{AM}(\Gamma):\Sigma_1\ne\Sigma_2\}\]
is open and dense in $\A$.
\end{rmk}

\begin{pf} We first show that $\mathfrak{C}$ is dense in $\A$. Let $\Gamma_0$ be any nullhomologous simple closed curve in $\partial M$. It suffices to show that
for any $\e>0$, there exists $\Gamma\in\mathfrak C$, such that $d(\Gamma, \Gamma_0)<\e$,  where by $d(\Gamma,\Gamma_0)<\e$ we mean that $\Gamma$ is
in a $\e$-neighborhood  of $\Gamma_0$ in the $\mathcal C^0$ topology (cf. Definition \ref{c0rmk}).

Assume that $g_0$ is the minimum genus of a surface bounded by $\Gamma_0$. By Theorem \ref{white2}, there exists an area minimizing
surface $\Sigma_{g_0}$ of  genus $g_0$. If $\Gamma_0\notin \B_{g_0}$, then $\Gamma_0$ bounds an absolutely area minimizing surface $\Sigma$ which has
to be different from $\Sigma_0$. In this case $\Gamma_0$ bounds at least two different stable minimal surfaces, namely $\Sigma_{g_0}$ and
$\Sigma$, so we can take $\Gamma=\Gamma_0$. Hence we can assume that $\Gamma_0\in\B_{g_0}$.

Given $g>g_0$,
by Theorem \ref{Adense}, there exists $\Gamma\in \A_g$, so that $d (\Gamma_0, \Gamma)<\e$ and for $\e$ small enough $\Gamma$ is homotopic to $\Gamma_0$. We denote the
annulus in $\partial M$ between $\Gamma$ and $\Gamma_0$ by $A_\e$, as we will need it later.

We claim that $\Gamma$ bounds two different stable minimal surfaces. Note first that since $\Gamma\in \A_g$ and $g> g_0$, it bounds a stable minimal surface
$\Sigma_1$ of genus bigger than $g_0$ (for example take an absolutely area minimizing surface that has minimum genus).
To show the existence of the second  surface, assume first that $\Gamma$ is nullhomotopic. Then, by the theorem of Meeks and Yau, Theorem \ref{meeksyau},  $\Gamma$ bounds a stable
minimal disk $\Sigma_2$.
Assume now that $\Gamma$ is not nullhomotopic, i.e. it bounds no disk. Let $\wt g_\Gamma$ be the minimum genus of a surface bounded by $\Gamma$. i.e.
\[ \wt g_\Gamma =\min\{g \ | \ \Gamma \text{ bounds a surface of genus g}\}. \]
We have that $0<\wt g_\Gamma \le g_0$; the second inequality being true because $\partial(\Sigma\cup A_\e)=\Gamma$, where $\Sigma$ is an area
minimizing surface of genus $g_0$ and with $\partial\Sigma=\Gamma_0$. This implies that the minimum genus of a surface bounded by $\Gamma$ is less than or equal to
$g_0$. By Theorem \ref{white2}, there exists an area minimizing surface $\Sigma_2$ of genus $\wt g_\Gamma$, which is therefore a stable minimal surface .

We have actually showed that the set
\[\mathfrak{C}'=\displaystyle\bigcup_{g>0}\{\Gamma\in\A_g \ | \ \Gamma \text{ bounds an area minimizing surface of genus }\le g-1\}\subset\mathfrak C\]
is dense in $\A$. Next we show that $\mathfrak{C}'$ is also open.

Let $\Gamma_0\in\mathfrak C'$. Then for some $g>0$, $\Gamma_0\in\A_g$ and it bounds a minimal surface of genus $\le g-1$. Since $\A_g$ is open, there
is an $\e>0$ such that for any $\Gamma\in\A$ with $d(\Gamma_0, \Gamma)<\e$, $\Gamma$ is homotopic to $\Gamma_0$ and belongs in $\A_g$.

Now, we can argue as before to show that any $\Gamma$ with $d(\Gamma, \Gamma_0)<\e$ bounds two minimal surfaces. In particular we have the following: since $\Gamma\in \A_g$, it bounds
an area minimizing surface $\Sigma_1$ of genus $\ge g$. If $\Gamma$ is nullhomotopic in $M$, then, as before, by Theorem \ref{meeksyau}, it bounds an area minimizing disk
$\Sigma_2$. If $\Gamma$ is not nullhomotopic, then we let
\[\wt g_\Gamma=\min\{g \ | \ \Gamma \text{ bounds a surface of genus g}\}.\]
We have that $0<\wt g_\Gamma\le g-1$ as before, and by Theorem \ref{white2}, there exists an area minimizing surface $\Sigma_2$ of genus $\wt g_\Gamma$.

This shows $\mathfrak{C}'$ is open and dense in $\A$. Since $\mathfrak{C}'\subset \mathfrak{C}$, the proof follows.
\end{pf}

\begin{rmk} Note that by \cite{Co1} and \cite{CE}, a generic simple closed curve in $\partial M$ bounds a unique area minimizing disk,
and similarly, a generic simple closed curve in $\partial M$ bounds a unique absolutely area minimizing surface in $M$. The result above shows that
when we relax the condition being area minimizing to just minimal, the situation is opposite. In other words, a generic simple closed curve bounds a
unique absolutely area minimizing surface, while a generic simple closed curve bounds more than one minimal surface. Hence the situations for minimal
and absolutely area minimizing surfaces are quite different.
\end{rmk}

Now, we generalize Peter Hall's examples to a more general setting. In \cite{Ha}, Peter Hall gave examples of nonembedded stable minimal surfaces in
the unit ball $\BB^3$ of $\BR^3$. To generalize this result to any strictly mean convex $3$-manifold $M$, we first show that there are disjoint
simple closed curves $\Gamma_1$ and $\Gamma_2$ in $\partial M$ and stable minimal surfaces $S_1$ and $S_2$ in $M$ with $\partial S_i = \Gamma_i$, $i=1,2$, that are
 not disjoint, i.e. $S_1\cap S_2 \neq \emptyset$.

\begin{thm} \label{intersectingminsurfaces}{\bf (Disjoint extreme curves may not bound disjoint minimal surfaces)}
Let $M$ be a strictly mean convex $3$-manifold. Then, there are stable minimal surfaces $S_1$ and $S_2$ in $M$ with $S_1\cap S_2 \neq
\emptyset$ and $\partial S_1 \cap \partial S_2 = \emptyset$. Moreover, these surfaces can be chosen to be disks.
\end{thm}

\begin{pf} We argue by contradiction using the technique in \cite[Theorem 3.2]{Co1}, along with Theorem \ref{nonuniquecurves}

Assume that the theorem is not true. By Theorem \ref{nonuniquecurves}, there exists $\Gamma_0\in\A$ and $\e>0$, such that any $\Gamma\in \A$ in an
$\e$-neighborhood of $\Gamma_0$ is homotopic to $\Gamma_0$ and bounds two distinct stable minimal surfaces. Taking $\e$ small enough,
$A_{\Gamma_0}=\{x\in\partial M: \dist(x, \Gamma_0)<\e\}$ is an annulus and we foliate this annulus by curves $\Gamma_t\in\A$, $t\in [-\e, \e]$. Each
$\Gamma_t$, by the choice of $\e$, bounds 2 stable minimal surfaces $\Sigma_t^+, \Sigma_t^-$. If $H_2(M,\BZ)=0$, then both $\Sigma_t^+$ and
$\Sigma_t^-$ are separating in $M$ and so there is a region {\em between} them. We let $N_t$ be that region, i.e. $N_t$ is such that $\partial
N_t=\Sigma_t^+\cup \Sigma_t^-$ as in \cite{Co1}. Then for $t_1< t_2$, $N_{t_1}\cap N_{t_2}=\emptyset$. This is because either $N_{t_1}\subset
N_{t_2}$, which is impossible since $\Gamma_{t_1}\cap\Gamma_{t_2} =\emptyset$, or else $\partial N_{t_1}\cap
\partial N_{t_2}\neq \emptyset$ which cannot hold because of our assumption that the theorem is false. Let $N=\cup_t N_t$. Then, for any $t\in [-\e, \e]$,
$|N_t|>0$ where $|N_t|$ now represents the volume of $N_t$, and $\sum |N_t|=| N| \le |M| <\infty$. This implies that $|N_t|\neq 0$ for at most
countably many $t$. But when $|N_t|=0$ and since $\Sigma^\pm_t$ are smooth, this implies $\Sigma_t^+=\Sigma_t ^-$. This is a contradiction as we have
assumed that for any $t$, $\Gamma_t$ bounds two distinct stable minimal surfaces $\Sigma_t^+$ and $\Sigma_t^-$. Hence for some $t_1<t_2$, we
must have $\partial N_{t_1}\cap \partial N_{t_2}\neq \emptyset$, i.e. $(\Sigma_{t_1}^+\cup\Sigma_{t_1}^-)\cap (\Sigma_{t_2}^+\cup\Sigma_{t_2}^-)\neq
\emptyset$ even  though $\Gamma_{t_1}\cap\Gamma_{t_2} =\emptyset$. The proof follows.

For the case $H_2(M,\BZ) \neq 0$, a similar argument from \cite{CE} works with a slight modification. If $H_2(M,\BZ) \neq 0$, then the surfaces
$\Sigma_t^+$ and $\Sigma_t^-$ may not be separating in $M$, and we cannot talk about the region $N_t$ {\em between} them. However, in this case, we
can restrict ourselves to a small neighborhood of $\Gamma_0$ in $M$, say $T_{\Gamma_0}$, which is a very thin solid torus with $T_{\Gamma_0} \cap
\partial M = A_{\Gamma_0}$. Now, the argument in the previous paragraph works, if we replace our ambient manifold $M$ with $T_{\Gamma_0}$ and the surfaces
$\Sigma^\pm_t$ with $S^\pm_t = \Sigma^\pm_t \cap T_{\Gamma_0}$. Then for any $t$, $S^+_t$ and $S^-_t$  are separating in $T_{\Gamma_0}$, and if we
assume that the theorem is  false, then $(S^+_t\cup S^-_t)\cap (S^+_s\cup S^-_s) = \emptyset$ for any $s\neq t$, since $S^\pm_t \subset
\Sigma^\pm_t$. Hence, by defining $N_t$ as the region between $S^+_t$ and $S^-_t$ in $T_{\Gamma_0}$, the whole summation argument goes through, and
we get a contradiction in this case, too. This proves the existence of intersecting stable minimal surfaces with disjoint boundaries.

Finally, we need to show that these surfaces can be chosen to be disks. For this we work as follows. We want to prove that there are stable
minimal disks $D_1$ and $D_2$ in $M$ with $D_1\cap D_2 \neq \emptyset$ and $\partial D_1 \cap \partial D_2 = \emptyset$. Let $p$ be a point in the
smooth part of $\partial M$. Let $\epsilon>0$ be sufficiently small so that $\overline{B}_\epsilon(p)$ is strictly mean convex (i.e. $\epsilon< \rho$
where $\rho$ is the convexity radius at $p$). Let $\gamma\in \B_0$ with $\gamma \subset \partial M\cap B_\epsilon(p)$ and let $E$ be the absolutely
area minimizing surface, which is a disk, with $\partial E=\gamma$. By construction, $E$ separates $M$ into two parts and we let $\Omega$ be the
``small'' part containing $p$. Then, $\Omega$ is topologically a $3$-ball, and any properly embedded surface would be separating in $\Omega$. Let
$\Gamma_0$ be in $\A_1$ and such that it has an annular  neighborhood $A_{\Gamma_0}$, for which $A_{\Gamma_0} \subset \Omega \cap \partial M$. Since
$\A_1$ is open by Theorem \ref{Agopen} (and by taking the annular neighborhood small enough), we can foliate $A_{\Gamma_0}$ with a collection
$\{\Gamma_t\}\in \A_1$ with $t\in [-\e, \e]$. Then for any $t\in [-\e, \e]$, the absolutely area minimizing surface $\Sigma_t$ with boundary
$\Gamma_t$ lies in $\Omega$, because $E$ is absolutely area minimizing surface and by \cite{Co1}, $E\cap\Sigma_t=\emptyset$. $\Sigma_t$ separates
$\Omega$ into two parts, say $\Omega_t^+$ (the part containing $p$) and $\Omega_t^-$. By construction, both $\Omega_t^+$ and $\Omega_t^-$ are mean convex.
Moreover, $\Gamma_t$ is nullhomotopic in both $\Omega_t^+$ and $\Omega_t^-$ as $\partial \Omega$ is a sphere. Hence, by the theorem of Meeks and Yau, Theorem \ref{meeksyau}, there are area
minimizing disks $D^+_t$ and $D^-_t$  in $\Omega_t^+$ and $\Omega_t^-$ respectively. Since $\Sigma_t$ has genus greater than or equal to one,
$D^\pm_t\neq \Sigma_t$, and hence $D^+_t\neq D^-_t$. This shows that for any $t\in [-\e, \e]$, $\Gamma_t$ bounds two stable minimal disks $D^+_t$ and
$D^-_t$ in $M$.

Assume that $D_t^\pm \cap D^\pm_s = \emptyset$ for any $t\neq s$. Then, we define $N_t$ as the region between $D^+_t$ and $D^-_t$ as in the first
paragraph. Notice that by assumption, $N_t\cap N_s =\emptyset$ and by construction, $|N_t|>0$ for any $t$. Then, the argument in the first paragraph
gives a contradiction as before and thus the proof follows.
\end{pf}

\begin{rmk} The result  of Theorem \ref{intersectingminsurfaces} is interesting even for the case when $M$ is a strictly mean convex $3$-manifold with $H_2(M,\BZ) = 0$. For such $M$, any two properly
embedded area minimizing surfaces $\Sigma_1$ and $\Sigma_2$ are disjoint, provided that their boundaries $\Gamma_1$ and $\Gamma_2$ in $\partial M$
are disjoint by the Meeks-Yau exchange roundoff trick \cite[Lemma 4.1]{Co1}. However, this theorem shows that if one relax the condition being area
minimizing to just being minimal, this is no longer true.
\end{rmk}

\begin{rmk}\label{manyintms} Note that the proof of Theorem \ref{intersectingminsurfaces} also shows that for any simple closed curve $\Gamma$ in $\partial M$,
it is possible to find {\em nearby} curves $\Gamma_1$ and $\Gamma_2$ such that $\Gamma_1\cap \Gamma_2 =\emptyset$ while $S_1\cap S_2 \neq \emptyset$
where $S_i$, $i=1,2$ is a stable minimal surface with boundary $\Gamma_i$. In particular, we can find such curves $\Gamma_1$ and $\Gamma_2$
in an $\e$-neighborhood of $\Gamma$, for any $\e$, such that the $\e$-neighborhood of $\Gamma$ in $\partial M$ is an annulus and moreover if $\Gamma$ is smooth, then we can find such curves $\Gamma_1, \Gamma_2$ so that they are also smooth. Furthermore, by the
proof of Theorem \ref{nonuniquecurves}, it is easy to see that we can take $S_2$ to be absolutely area minimizing and $S_1$ to be area minimizing of
genus $g$, where $g$ is less than or equal to the minimum genus of a surface bounded by $\Gamma$.
\end{rmk}

Now, we need a version of a bridge principle to construct a nonembedded stable minimal surface in a strictly mean convex $3$-manifold whose boundary is
a simple closed curve.

\begin{lem}\label{bridgeprincipleforextremecurves} {\bf (Bridge Principle for Extreme Curves)}

Let $M$ be a strictly mean convex $3$-manifold, and $\Gamma_1$ and $\Gamma_2$ be two disjoint smooth simple closed curves bounding two
stable minimal surfaces $S_1$ and $S_2$ respectively. Let $\alpha$ be a smooth path in $\partial M$ connecting $\Gamma_1$ and $\Gamma_2$, as in
\emph{Definition \ref{a2}}. Then,  for any $\e>0$, there is a stable minimal surface $S_1\sharp_\alpha S_2$ which is $\e$-close to $S_1\cup
S_2 \cup \alpha$.
\end{lem}

\begin{pf} This lemma is a direct application of the Meeks and Yau Bridge principle \cite[Theorem 7]{MY3}, noting only that in their proof one can pick the bridge to lie on $\partial M$, since $M$ is strictly convex.
\end{pf}

Finally, by using the previous result, and by the bridge principle above, we generalize Peter Hall's examples of nonembedded stable minimal surfaces
in $\BB^3$ to any strictly mean convex $3$-manifold $M$. In particular, we show that for any strictly mean convex $3$-manifold $M$, there exists a
simple closed curve $\Gamma$ in $\partial M$ which bounds a nonembedded stable minimal surface $S$ in $M$.

\begin{thm}\label{nonembeg} {\bf (Nonembedded Stable Minimal Surfaces)}

Let $M$ be a strictly mean convex $3$-manifold. Then, there is a simple closed curve $\Gamma$ in $\partial M$ such that $\Gamma$ bounds a nonembedded
stable minimal surface $S$ in $M$. Furthermore  there is a simple closed curve $\Gamma$ in $\partial M$ such that $\Gamma$ bounds a nonembedded
stable minimal disk.
\end{thm}

\begin{pf}
Take $S_1, S_2$ as in Theorem \ref{intersectingminsurfaces}, so that $\partial S_1,\partial S_2$ are smooth (cf. Remark \ref{manyintms}). Then $S_1,
S_2$ intersect transversally, (using a local graphical representation and the the maximum principle one sees that they cannot be tangent). Now, apply
Lemma \ref{bridgeprincipleforextremecurves} to construct a minimal surface that is $\e$-close to $S_1\cup S_2\cup\alpha$, where $\alpha$ is a curve
connecting $\partial S_1$ and $\partial S_2$. Since $S_1, S_2$ intersect transversally, for $\e$ small enough the new surface has at least one self
intersection. For the disk case, take the surfaces as disks as in Theorem \ref{intersectingminsurfaces}. The proof follows.
\end{pf}

\begin{rmk} \label{manycurvesremark} Note that using Remark \ref{manyintms}, we actually have many curves satisfying Theorem \ref{nonembeg}. In particular, given a smooth curve
$\Gamma_0\in\A$ and $\e$, such that  the $\e$-neighborhood of $\Gamma_0$ in $\partial M$ is an annulus, there exists $\Gamma$ in the
$\e$-neighborhood of $\Gamma_0$,  satisfying Theorem \ref{nonembeg}.
\end{rmk}

\begin{rmk} Again, when we compare this result with the area minimizing case, we see that this cannot happen in the area minimizing case. By the
regularity results of \cite{ASSi}, any area minimizing surface $\Sigma$ in $M$ bounded by a simple closed curve $\Gamma$ in $\partial M$
must be smooth in the interior. Hence, if we relax the condition being area minimizing to being just minimal, we see that this
is no longer true.
\end{rmk}

\section{Applications to Curves in $\BR^3$}

In this section we apply our results  to simple closed curves in $\BR^3$. So far, we discussed the simple closed curves in the
boundary of a strictly mean convex $3$-manifold, and proved several results about the area minimizing surfaces  in the manifold that these curves bound. Now, we
change our focus to the simple closed curves in $\BR^3$ and give extensions of these results to this case.

We use the notation from the previous sections. In particular, $\A$ denotes the space of simple closed curves in $\BRR$. $\A_g\subset \A$ denotes the
space of simple closed curves in $\BR^3$ such that the minimum genus of the absolutely area minimizing surfaces that they bound is greater than or
equal to $g$. $\B_g$ denotes the space of simple closed curves in $\partial M$ such that the minimum genus of the absolutely area minimizing surfaces
that they bound is exactly $g$. i.e. $\B_g=\A_g-\A_{g+1}$. Note that $\A_0=\A$ because of the existence of absolutely area minimizing surfaces with boundary any given
simple closed curve in $\BR^3$ \cite{Fe}. Furthermore we let $\A^k$ denote the $\C^k$ smooth simple closed curves in $\BR^3$ and we let $\A^k_g$ be $\A^k\cap\A_g$.

First, we extend Lemma \ref{Agopen}  and  Theorem \ref{Adense} to the case of curves in $\BR^3$.
\begin{thm}\label{AgopenR3}
For any $g\geq 0$, $\A_g$ is open and dense in $\A$ equipped with the $\C^0$ topology. Moreover, for any $k\geq 0$, $\A^k_g$ is open and dense in $\A^k$ equipped with
$\C^0$ topology.
\end{thm}

\begin{pf} First, we show that $\A_g$ is open in $\A$. The proof of this fact is the same as the proof of Lemma \ref{Agopen}, with the only difference that here $\mathcal U _{\Gamma_0}(\e)$ is an $\e$-neighborhood  of $\Gamma_0$ in $\BR^3$ with respect to the $\mathcal C^0$ topology; whereas in the proof of Lemma \ref{Agopen} we restricted this neighborhood in $\partial M$; the boundary of the mean convex manifold $M$. Restricting ourselves now to $\C^k$ curves shows that $\A^k_g$ is open in
$\A^k$ with respect to the $\C^0$ topology.

Now, we show that $\A_g$ is dense in $\A$. Let $\Gamma$ be a simple closed curve in $\BRR$. Let $\Sigma$ be the absolutely area minimizing
surface in $\BRR$ with $\partial \Sigma = \Gamma$. Then, by \cite[Page 159, Corollary 1]{MY3} , for a sequence $\Sigma_i$ of proper subsurfaces
($\Sigma_i \varsubsetneq \Sigma$) with $\Sigma_i\to \Sigma$, there are strictly mean convex neighborhoods $N_i$ of $\Sigma_i$ in $\BRR$ such that
$\Gamma_i :=\partial \Sigma_i \subset \partial N_i$.

For a set $X$, let $\A(X)$ denote the space of simple closed curves in $\partial X$ and let $\A_g(X)$ denote the space of simple closed curves in $\partial X$ such that the minimum genus of the absolutely area minimizing surfaces that
they bound in $X$ is greater than or equal to $g$. With this notation, $\A_g=\A_g(\BR^3)$. Since $N_i$ is strictly mean convex, by Theorem \ref{Adense},
$\A_g(N_i)$ is dense in $\A(N_i)$, the space of simple closed curves in $\partial N_i$. Hence, for each $i$, there is a sequence
$\{\alpha^j_i\}_{j\in\N}$ of simple closed curves in $\A_g(N _i)$ such that $\alpha^j_i \stackrel{j\to\infty}\longrightarrow \Gamma_i$, with respect to the $\mathcal C^0$ topology, where recall that $\Gamma_i=\partial \Sigma_i\subset \partial
N_i$.

Let $S^j_i$ be an absolutely area minimizing surface in $N_i$ with boundary $\partial S^j_i=\alpha^j_i$ and with $\genus (S^j_i)\geq g$. Notice that
$S^j_i$ is absolutely area minimizing in $N_i$, and this does not necessarily imply that it is also absolutely area minimizing in $\BR^3$. We
claim however that for any $i>0$, there is a $j_0>0$, such that for any $j\ge j_0$, $S^j_i$ is also absolutely area minimizing in $\BR^3$. Hence, when we prove the claim, the
density of $\A_g$ in $\A$ follows.

First, note that $\Sigma_i$ is the unique absolutely area minimizing surface in $\BR^3$ with $\partial \Sigma_i = \Gamma_i$. This because if there
was another absolutely area minimizing surface $\Sigma'_i$ in $\BR^3$ with $\partial \Sigma_i' = \Gamma_i$, then $\Sigma'=(\Sigma\setminus\Sigma_i)\cup
\Sigma_i'$ would be an absolutely area minimizing surface with $\partial \Sigma'=\Gamma$, since $|\Sigma_i|=|\Sigma_i'|$. However, $\Sigma'$ then would have a
singularity along $\Gamma_i$, and this contradicts to the regularity theorem for absolutely area minimizing surfaces \cite{Fe}.

Now, fix $i_0>0$. Assume that the above claim is not true and hence there is a sequence $\{j_k\}_{k\in\N}$ with  $j_k\stackrel{k\to\infty}{\longrightarrow}\infty$, such that $S^{j_k}_{i_0}$ is not absolutely area minimizing surface in $\BR^3$ for any $k\in\N$. Let $T^{j_k}_{i_0}$ be the absolutely area
minimizing surface in $\BR^3$ with $\partial T^{j_k}_{i_0}=\alpha^{j_k}_{i_0}$. Since $\alpha^{j_k}_{i_0}\stackrel{k\to\infty}{\longrightarrow}\Gamma_{i_0}$, and $\Sigma_{i_0}$ is the unique
absolutely area minimizing surface in $\BR^3$ with $\partial \Sigma_{i_0}=\Gamma_{i_0}$, then after passing to a subsequence if necessary,
$T^{j_k}_{i_0}\stackrel{k\to\infty}{\longrightarrow}\Sigma_{i_0}$, where the convergence is with respect to measure and to Hausdorff distance. Since $N_{i_0}$ is strictly mean convex, there is a neighborhood $N^\delta$ of $N_{i_0}$ in $\BR^3$ such
that the nearest point projection map $\pi:N^\delta \to N_{i_0}$ does not increase the area, i.e. $|\pi(S)|\leq |S|$ for a surface $S$ in
$N^\delta$. Since, $T^{j_k}_{i_0}\stackrel{k\to\infty}{\longrightarrow}\Sigma_{i_0}$ in Hausdorff distance, there exists a $k_0>0$ such that for any $k\ge k_0$, $T^{j_k}_{i_0}\subset N^\delta$. Hence,
$|\pi(T^{j_k}_{i_0})|\leq |T^{j_k}_{i_0}|$, and since $\pi(T^{j_k}_{i_0})\subset N_{i_0}$ with $\partial \pi(T^{j_k}_{i_0})=\alpha^{j_k}_{i_0}$, $|S^{j_k}_{i_0}|\leq
|\pi(T^{j_k}_{i_0})|\leq |T^{j_k}_{i_0}|$. This shows that $S^{j_k}_{i_0}$ is also absolutely area minimizing in $\BR^3$ for $k\ge k_0$, and the claim follows with $j_0=j_{k_0}$.

This shows that for any $i>0$, there is a $j_0>0$ such that $\alpha^j_i\in \A_g(\BR^3)$ for all $j\ge j_0$. Hence, $\A_g$ is not only open, but also dense in $\A$ in
$\C^0$ topology. Notice that if the initial curve $\Gamma$ is smooth ($\C^k$ smooth), then we can also take the curves $\alpha^j_i$ to be as smooth as $\Gamma$ by Lemma \ref{thinhandle}. This implies that $\A^k_g$ is dense in $\A^k$
in the $\C^0$ topology.
\end{pf}

\begin{rmk} \label{rmkAgopenR3} Note that from the proof of Theorem \ref{AgopenR3} we can conclude that $\A^k_g$ is open in $\A^k$ with respect to the $\C^k$ topology. However, the statement that $\A^k_g$ is dense in
$\A^k$ with respect to the $\C^k$ topology is far from being true. To see that, let $\gamma$ be a $\C^2$ simple closed curve in $\BR^3$, that is furthermore contained in a plain. Then, by \cite{Me}, there is a $\C^2$
neighborhood $N(\gamma)$ of $\gamma$ such that for any $\gamma'\in N(\gamma)$, there exists a unique minimal surface $S$ in $\BR^3$ with $\partial S
=\gamma'$, and $S$ is a graph over the plane (hence a disk). This shows that $\gamma\in \inr(\B_0)$, and thus $\A_1$ (or any $\A_g$ for $g\ge 1$) is not dense in $\A$ in the $\C^2$ topology.
The main reason why the techniques of this paper do not generalize to the smooth topology is that the horn surgery produces curves that are ``close'' to the original curve in the
$\C^0$ topology but not in any $\C^k$ topology for $k\ge 2$.
\end{rmk}

Next, we apply the generic non-uniqueness result of  Section \ref{minsurfaces section} to  curves in $\BR^3$. Let $\mathfrak C$ be the space of
simple closed curves in $\BR^3$ bounding more than one embedded stable minimal surface. We show that $\mathfrak{C}$ is dense in $\A$ with respect to
the $\C^0$ topology.

\begin{thm}\label{nonuniqueR3}{\bf (Curves bounding more than one minimal surface are dense)}
Let $\A$ be the space of simple closed curves in $\BR^3$ equipped with the $\C^0$ topology. Let $\mathfrak{C}\subset \A$ represent the simple closed
curves in $\partial M$ bounding more than one embedded stable minimal surface in $\BR^3$. Then, $\mathfrak{C}$ is dense in $\A$ with respect to the
$\C^0$ topology. Moreover, for any $k\geq 0$, $\CC\cap\A^k$ is dense in $\A^k$ with respect to the $\C^0$ topology.
\end{thm}

\begin{pf} Let $\wh{\CC}=\{\Gamma\in\A:\exists \Sigma_1\in \text{AAM}(\Gamma), \Sigma_2\in \text{AM}(\Gamma):\Sigma_1\ne\Sigma_2\}$, where similarly to
Theorem \ref{nonuniquecurves}, we let AAM$(\Gamma)$ and AM$(\Gamma)$ be the set of Absolutely Area Minimizing and Area minimizing surfaces
respectively bounded by $\Gamma$. Clearly $\wh{\CC}\subset\CC$. Hence, if we show that $\wh{\CC}$ is dense in $\A$, we are done.

Let $\Gamma , \Sigma , \Sigma_i$ and $N_i$ be as in the proof of  Theorem \ref{AgopenR3}, i.e. $\Gamma\in\A$, $\Sigma\in $AAM$(\Gamma)$,
$\Sigma_i\subsetneq \Sigma$, with $\Sigma_i\to\Sigma$ and $N_i$ are strictly mean convex neighborhoods of $\Sigma_i$ such that
$\Gamma_i:=\partial\Sigma_i\subset\partial N_i$. Analogously with $\A(X)$, as defined in the proof of Theorem \ref{AgopenR3}, for a set $X$, we let
$\CC(X)$ represent the curves in $\A(X)$ that bound more than one stable minimal surface in $X$. Since $N_i$ is strictly mean convex, Theorem
\ref{nonuniquecurves} implies that $\CC(N_i)$ is open and dense in $\A(N_i)$. The embedded minimal surfaces in $N(\Sigma_i)$ are also minimal in
$\BRR$. The smooth case is similar. The proof follows.
\end{pf}

\begin{rmk} Notice that  Theorem \ref{nonuniquecurves} says not only density, but also genericity (containing open dense subset) of $\CC$ in $\A$,
when the ambient manifold $M$ is strictly mean convex and the curves are in $\partial M$. Here, genericity of $\CC$ for the curves in $\BR^3$ cannot
be proved with the techniques of Theorem \ref{nonuniquecurves} as White's result in \cite{Wh2} (cf. Theorem \ref{white2}) is for strictly mean convex manifolds; this theorem
 gives us an embedded minimal surface of smallest genus when the manifold is strictly mean convex.

On the other hand, instead of Theorem \ref{white2}, one might want to consider \cite[Corollary 2.1]{Jo} to get a minimal representative of the
smallest genus surface, and generalize these arguments as it applies to $\BR^n$. However, Jost's result gives us an \underline{immersed} (not
necessarily embedded) least area representative of the smallest genus surface the curve bounds (which is always a disk in our case). Hence, we cannot
use it here, as we are looking for embedded minimal surfaces.
\end{rmk}

Finally, we will generalize Theorem \ref{nonembeg} to  simple closed curves in $\BR^3$.

\begin{thm} For any simple closed curve $\Gamma$ in $\BRR$ and for any $\e>0$, there is a smooth simple closed curve $\Gamma'$ in $N_\e(\Gamma)$; an $\e$-neighborhood of $\Gamma$, such that
$\Gamma'$ bounds a nonembedded stable minimal surface in $\BRR$.
\end{thm}

\begin{pf} Let $\Gamma ,\Gamma_i, \Sigma , \Sigma_i$ and $N_i$ be as in the proof of Theorems \ref{AgopenR3} and \ref{nonuniqueR3}, i.e. $\Gamma\in\A$, $\Sigma\in $AAM$(\Gamma)$, $\Sigma_i\subsetneq \Sigma$, with $\Sigma_i\to\Sigma$ and $N_i$ are strictly mean convex neighborhoods of $\Sigma_i$ such that $\Gamma_i:=\partial\Sigma_i\subset\partial N_i$. Since $N_i$ is
strictly mean convex, Theorem \ref{nonembeg} (see also Remark \ref{manycurvesremark}) implies that for any $\e>0$, there is a simple closed curve $\Gamma'_i\subset N_\e(\Gamma_i)\cap\partial
N_i$, where $N_\e(\Gamma_i)$ is an $\e$-neighborhood of $\Gamma_i$ in the $\C^0$ topology,  such that $\Gamma'_i$ bounds a nonembedded stable minimal surface in $N_i$. Since any minimal surface in $N_i$ is also
minimal in $\BRR$ and since, by construction, the curves can be taken smooth, the proof follows.
\end{pf}

\begin{rmk} The result above shows that the simple close curves in $\BR^3$ bounding nonembedded minimal surfaces, say $\mathfrak{D}$, are very abundant. Note that our
result does not imply the density of $\mathfrak{D}$ in $\A$. Even though the curve $\Gamma'$ is in $\e$ neighborhood of $\Gamma$, it is not
$\C^0$-close to $\Gamma$ (see proof of Theorem \ref{nonembeg}).

On the other hand, even though we cannot say whether $\mathfrak{D}$ is dense in $\A$ in the $\C^0$ topology or not, it is easy to see that
$\mathfrak{D}\cap\A^2$ is not dense in $\A^2$ in the $\C^2$ topology. By \cite{EWW}, any minimal surface $\Sigma$, bounded by a simple closed curve $\Gamma$ in $\BR^3$ with total curvature
$k(\Gamma)\leq 4\pi$ must be embedded. Hence,  a simple closed curve $\Gamma$, satisfying $k(\Gamma)<4\pi$ has a
$\C^2$-neighborhood $N(\Gamma)$ such that for any $\Gamma'\in N(\Gamma)$ the following holds: any minimal surface $\Sigma'$ in $\BR^3$ bounded by $\Gamma'$ is embedded.
Similarly, the result  by Meeks \cite{Me} mentioned in the previous Remark \ref{rmkAgopenR3}  gives such a neighborhood of a curve $\Gamma$. This shows that $\mathfrak{D}$ is not dense in $\A^2$ in the
$\C^2$ topology.
\end{rmk}

\section{Final Remarks}

In the last section, for curves in $\R^3$, we considered the questions of openness and density not only for simple closed curves but also for smooth curves. The same thing can be done for curves on the boundary of a mean convex manifold $M$ and similar questions can be considered in the smooth category. Recall that $\A_g$ denotes the space of simple closed curves in $\partial M$ such that
the minimum genus of the absolutely area minimizing surfaces that they bound is greater than or equal to $g$. $\B_g$ denotes the space of simple
closed curves in $\partial M$ such that the minimum genus of the absolutely area minimizing surfaces that they bound is exactly $g$. Also, $\A$
denotes the space of nullhomologous simple closed curves in $\partial M$ equipped with the $\mathcal C^0$ topology. A natural extension is to
consider the results of this paper in the smooth category. In other words, let  $\A^k$ be equal to $\A\cap C^k$, the space of nullhomologous curves
that are $\C^k$ smooth, considered again with the $\mathcal{C}^0$ topology.


Let also $\A^k_g =\A_g\cap C^k$ and $\B^k_g =\B_g\cap C^k$ be the $\C^k$ smooth curves in $\A_g$ and in $\B_g$ respectively. Then, it is easy to show
that $\A^k_g$ is open in $\A^k$ by using the techniques of this paper (Lemma \ref{Agopen}). Furthermore, if $\partial M$ is smooth then the fact that
$\A_g$ is open and dense in $\A$ (Lemma \ref{Adense}), and $\A^k$ is dense in $\A$ implies that $\A^k_g$ is dense in $\A$. If $\partial M$ is not
smooth we only get that $\A^k_g$ is dense in $\A^k$.

%

Another  question would be whether the similar results are true in the mean convex setting instead of strictly mean convex manifolds. Unfortunately,
most of the results in Section 3, like $\A_g$ is open dense in $\A$, are not true for mean convex $3$-manifolds. The main obstacle is the following.
Let $M$ be a mean convex $3$-manifold and $S$ be a subsurface in $\partial M$ which is a minimal surface. Take a sufficiently small disk $D$ in $S$
such that $D$ is the absolutely area minimizing surface for its boundary. Then, for any simple closed curve in the interior of $D$, the absolutely
area minimizing surface would be the disk it bounds in $D$. Hence, any simple closed curve $\Gamma$ in $\inr(D) \subset \partial M$ would
automatically be in $\B_0$ and therefore there is no way to approximate this $\Gamma$ with the curves in $\A_g$, for $g\geq 1$, since there is an
open neighborhood of $\Gamma$ in $A$ which lies entirely in $\B_0$. This example shows that Theorem \ref{Adense} ($\A_g$ is dense in $\A$) is not
true for mean convex $3$-manifolds whose boundary contains a minimal subsurface. For an extreme example, if $M$ is a mean convex manifold whose
boundary is an absolutely area minimizing sphere, then clearly $\A=\B_0$. The
reason for this is that the horn surgery lemma fails when the absolutely area minimizing surface it bounds is in the boundary of the manifold. What
in particular fails, is that we cannot apply White's bridge principle, Theorem \ref{white4}. In \cite{Wh4}, there is no restriction on $M$ being strictly mean
convex, however it is required that the arc $\a$ is not tangential at its endpoints to the tangent halfspace of the area minimizing surface. Of
course this is not the case when both $\a$ and the absolutely area minimizing surface lie in $\partial M$.  Hence, the relevant results depending on
Theorem \ref{white4}, would not be valid for mean convex $3$-manifolds.

%

In Section 4, we showed that the case of absolutely area minimizing surfaces  and that of minimal surfaces are quite different when bounding extreme curves.
For a strictly mean convex $3$-manifold $M$, while generically a simple closed curve in $\partial M$ bounds a unique absolutely area minimizing
surface in $M$ \cite{Co1}, \cite{CE}, the situation is opposite for minimal surfaces, as generically a simple closed curve in $\partial M$ bounds
more than one minimal surface in $M$ (Theorem 4.1). Also, for $H_2(M,Z)=0$, while absolutely area minimizing surfaces with disjoint boundaries in
$\partial M$ are also disjoint, this is no longer true for minimal surfaces. In Theorem 4.2, we proved that for any strictly mean convex $3$-manifold
$M$, there are disjoint simple closed curves in $\partial M$, bounding intersecting stable minimal surfaces.

It might be interesting to study the space of curves bounding nonembedded minimal surfaces as described in Section 4. In other words, we showed the
existence of such extreme curves for any strictly mean convex $3$-manifold $M$ in Theorem \ref{nonembeg}. However, a condition on the curve to
guarantee the embeddedness of all minimal surfaces it bounds would be very interesting. In our construction for the curves bounding nonembedded
minimal surfaces, we used two intersecting minimal surfaces in $M$ which have disjoint boundaries in $\partial M$. Then, we used a bridge between
them to get a nonembedded minimal surface. This implies that for the final boundary curve $\Gamma$ and near the bridge, the ratio between the extrinsic
distance (in $\partial M$) and the intrinsic distance (in $\partial M$) gets very large. Maybe, a bound (of course, depending on $M$) on this ratio
might be a good condition to guarantee the embeddedness of all minimal surfaces which $\Gamma$ bounds in $M$.

Furthermore, for any $g>0$, Almgren and Thurston showed existence of unknotted simple closed curves $\Gamma_g$ in $\BR^3$ with the property
that the absolutely area minimizing surface $\Gamma_g$ bounds has genus at least $g$ \cite{AT}. Indeed, they showed that for any $g>0$, there are
curves $\Gamma_g$ such that any embedded surface they bound in their convex hull of the curve must have genus at least $g$. Since any minimal surface
with boundary $\Gamma$ must be in the convex hull of $\Gamma$, and any absolutely area minimizing surface is embedded, this automatically implies
that the absolutely area minimizing surface $\Gamma_g$ bound must have genus at least $g$. In this paper's terminology, this means Almgren and
Thurston showed that for any $g>0$, $\A_g \neq \emptyset$. However, their examples are not extreme curves. Here, one of the corollaries of Theorem
\ref{Adense} is that for any convex body $\Omega$ in $\BR^3$ and for any $g>0$, there is a curve $\Gamma_g \in \partial \Omega$ such that the
absolutely area minimizing surface, which $\Gamma_g$ bounds, has genus at least $g$.

-

\end{document}